\documentclass{article}

\usepackage{microtype}
\usepackage{natbib}
\usepackage{graphicx}
\usepackage{subfigure}
\usepackage{booktabs} 
\usepackage{subfigure}
\usepackage{epsfig,epsf,fancybox}
\usepackage{amsmath}
\usepackage{mathrsfs}
\usepackage{amssymb}
\usepackage{graphicx}
\usepackage{color}
\usepackage{nicefrac}
\usepackage{multirow}
\usepackage{paralist}
\usepackage{verbatim}
\usepackage{galois}
\usepackage{algorithm}
\usepackage{algorithmic}
\usepackage{bbm}
\usepackage{mathtools}
\usepackage{stmaryrd}
\usepackage{bm}
\usepackage{amsmath}

\newtheorem{theorem}{Theorem}[section]

\newtheorem{lemma}[theorem]{Lemma}

\newtheorem{proof}{Proof}[section]

\newenvironment{smallpmatrix}
    {\left(
    \begin{smallmatrix}}
    {\end{smallmatrix}
    \right)
    }
\newcommand{\R}{\mathbb{R}}

\newcommand{\Diag}{\textnormal{Diag}\,}
\newcommand{\x}{\mathbf x}
\newcommand{\y}{\mathbf y}
\newcommand{\s}{\mathbf s}
\newcommand{\bd}{\mathbf d}

\newcommand{\bb}{\mathbf b}

\newcommand{\z}{\mathbf z}
\newcommand{\w}{\mathbf w}
\newcommand{\p}{\mathbf p}

\newcommand{\argmin}{\mathop{\rm argmin}}
\usepackage{hyperref}

\usepackage[accepted]{icml2021}

\icmltitlerunning{Fast algorithms for SPG-LS}

\begin{document}

\twocolumn[
\icmltitle{Fast Algorithms for Stackelberg Prediction Game with Least Squares Loss}



\begin{icmlauthorlist}
\icmlauthor{Jiali Wang}{sds}
\icmlauthor{He Chen}{math}
\icmlauthor{Rujun Jiang}{sds}
\icmlauthor{Xudong Li}{sds}
\icmlauthor{Zihao Li}{math}
\end{icmlauthorlist}
\icmlaffiliation{sds}{School of Data Science, Fudan University, China}
\icmlaffiliation{math}{School of Mathematical Sciences, Fudan University, China}

\icmlcorrespondingauthor{Rujun Jiang}{rjjiang@fudan.edu.cn}

\icmlkeywords{Machine Learning, ICML}

\vskip 0.3in
]



\printAffiliationsAndNotice{}  

\begin{abstract}
The Stackelberg prediction game (SPG) has been extensively used to model the interactions between the learner and data provider in the training process of various machine learning algorithms. Particularly, SPGs played prominent roles in cybersecurity applications, such as intrusion detection, banking fraud detection, spam filtering, and malware detection. Often formulated as NP-hard bi-level optimization problems, it is generally computationally intractable to find global solutions to SPGs.
As an interesting progress in this area, a special class of SPGs with the least squares loss (SPG-LS) have recently been shown polynomially solvable by a bisection method.
However, in each iteration of this method, a semidefinite program (SDP) needs to be solved. The resulted high computational costs prevent its applications for large-scale problems.
In contrast, we propose a novel approach that reformulates a SPG-LS as a single SDP of a similar form and the same dimension  as those solved in the bisection method. Our SDP reformulation is, evidenced by our numerical experiments, orders of magnitude faster than the existing bisection method. We further show that the obtained SDP can be reduced to a second order cone program (SOCP). This allows us to provide real-time response to large-scale SPG-LS problems. Numerical results on both synthetic and real world datasets indicate that the proposed SOCP method is up to 20,000+ times faster than the state of the art.
\end{abstract}

\section{Introduction}
In the big data era, machine learning (ML) algorithms have been extensively used to extract useful information from data and found numerous applications in our daily life. In certain areas, such as cybersecurity, the nature of applications requires high robustness of ML algorithms against adversarial attacks. A typical scenario would be the training data that the ML algorithms or the learner relied on is deliberately altered by a malicious adversary.
In this case, the key assumption for the success of ML algorithms, i.e., the stationarity of data or equivalently the independent and identically distributed (i.i.d.) assumption, fails to hold.
To alleviate this difficulty, researchers have proposed various game theoretic approaches \cite{bruckner2011stackelberg,tong2018adversarial,vorobeychik2018adversarial,bishop2020optimal} to model the strategic interactions between the learner and the attacker -- in our case, the adversarial data provider.

In practice, there are also many applications that the learner and the data providers are not entirely antagonistic, where the data providers often manipulate the data only for their own interests. Introduced by \citet{bruckner2011stackelberg}, the SPG is used to model the interactions between the learner and the data provider as a two-players non-zero-sum sequential game for such cases. In the SPG, the learner is regarded as the leader who makes the first move to commit to a predictive model without knowing the strategy of the data provider (or the follower). Then, the data provider, based on the available information of the learner's predictive model, selects his costs-minimizing strategy to modify the data against the learner. Under the rationality assumption of both the learner and data provider, \citet{bruckner2011stackelberg} introduced the notion of Stackelberg equilibrium as the optimal strategy of the SPG and  proposed to find it via solving a corresponding bi-level optimization problem. Particularly, the bi-level optimization problem minimizes the prediction loss from the learner's perspective under the constraint that the data has been optimally modified from the data provider's perspective.
Since then, SPGs have received a lot of attention in the literature
\cite{shokri2012protecting,zhou2016modeling,wahab2016stackelberg,papernot2018towards,naveiro2019gradient,zhou2019survey}. Unfortunately,  bi-level optimization problems are generally NP-hard \cite{jeroslow1985polynomial} and their optimal solutions are intrinsically difficult to obtain, which severely limit the applicability of SPGs in real world use cases.

Recently,  \citet{bishop2020optimal} made the first step to globally solve a special subclass of SPGs.  Specifically, they restricted their interests on SPGs with least squares loss (SPG-LS) (i.e., all the loss functions for  the learner and data providers are  the least squares).
They further reformulated the SPG-LS into a quadratically constrained quadratic fractional program that can be solved via a bisection method.
In each iteration of their bisection method, a nonconvex quadratically constrained quadratic program (QCQP) needs to be \textit{exactly}\ solved.
Fortunately, by using the celebrated S-lemma \cite{yakubovich1971s,polik2007survey,xia2016s}, the optimal solutions to the nonconvex QCQPs can be obtained via solving their semidefinite programming (SDP) relaxations \cite{vandenberghe1996semidefinite}.
However, the number of bisection searches is often of several tens in practice.
This, together with heavy computational costs of solving each SDP, makes the bisection method far less attractive especially for large-scale problems.
Moreover, the requirement for \textit{exactly} solving each SDP is too strong for large-scale problems, even armed with powerful academic and commercial solvers.
Theoretically speaking, given the accumulation of these inaccuracy, the convergence of the bisection method with inexact SDPs' solutions remains unknown.
More importantly, this accumulated inexactness may finally result unstable algorithmic performances, which prevents its applications in the area of security.

In this paper, we aim to resolve the above mentioned scalability and stability issues of the bisection method for the SPG-LS.
For this purpose, we start by re-examining the quadratic fractional program (QFP) considered in \citet{bishop2020optimal}.
By using the S-lemma in a slightly different way, we show that the QFP can be directly reformulated into an SDP of almost the same problem size as the ones in the bisection method.
Furthermore, we prove that there always exists an optimal solution for our SDP and an optimal solution to the SPG-LS can be recovered from the optimal solution of our SDP.
It thus implies that the bisection steps are unnecessary, i.e., to solve the SPG-LS, one only needs to solve a single SDP.
This novel reformulation outperforms the bisection method by a significant margin as the latter involves solving a series of  SDPs with similar problem sizes.
Surprisingly, we can take a step  further in accelerating our method. By carefully investigating the intrinsic structures of the proposed SDP, we show our single SDP reformulation can be further reduced into a second order cone program (SOCP) \cite{alizadeh2003second}, which can be solved much more efficiently than SDPs in general.
More specifically, we apply two congruence transformation for the linear matrix inequality (LMI) in our SDP.
 The second congruence in fact explores
a simultaneous diagonalizabiliy of submatrices for the three matrices in the LMI after the first congruence transformation. Then by using a generalized Schur complement, we demonstrate that our SDP can be further reformulated as an SOCP with a much smaller size.
The main cost in our reformulation is a spectral decomposition for the data matrix that is cheap even for large instances.
Moreover, solving our SOCP reformulation is even cheaper than one spectral decomposition.
Hence  our SOCP method is much faster than our single SDP method.



We summarise our contributions as follows:
\begin{itemize}
        \item We derive a single SDP reformulation for the SPG-LS, while the state of the art needs to solve dozens of SDPs with similar problems sizes.
        \item We further derive an SOCP reformulation  with a much smaller dimension than our single SDP.
        \item We propose two efficient ways to recover an optimal solution for the SPG-LS  either from a rank-1 decomposition of the dual solution of our single SDP or by solving a linear system with an additional equation.
        \item We show that our methods significantly improve the state of the art by numerical experiments on both synthetic and real data sets.
\end{itemize}

\section{Preliminaries}
In this section, we formalize the SPG-LS problem by adapting the same setting as in \citet{bishop2020optimal}. A brief review of Bishop et al.'s bisection method will also be provided.

Similar as in \citet{bishop2020optimal}, we assume that the learner has access to a sample $S = \left\{(\x_i, y_i, z_i) \right\}_{i=1}^m$ with each $\x_i \in\R^n$ been the input example, $y_i$ and $z_i$ been the output labels of interests to the learner and the data provider, respectively. The samples are assumed to be realizations of $(\x, y, z)$ following some fixed but unknown distribution ${\cal D}$. The learner then aims to train a linear predictor $\w\in\R^n$ based on $S$, i.e., to predict correctly label $y$ when supplied with $\x$. In the SPG-LS, being aware of the learner's predictor $\w$, the goal of the adversarial data provider is to fool the learner to predict the label $z$ by modifying the input data $\x$ to $\hat \x$ while maintaining low manipulation costs. Here, we follow \citet{bishop2020optimal} to model the modifying costs from $\x$ to $\hat \x$ as $\gamma \|\x - \hat\x\|^2$ with some positive parameter $\gamma$.

To find the Stackelberg equilibrium of the above SPG-LS, we formulate in the following the corresponding bi-level optimization problem. Given the disclosed predictor $\w\in\R^n$ and the training set $S$, the data provider described above aims to solve the following optimization problems:
$$
\x_i^* = \argmin_{\mathbf{\hat{x}_i}}\ \|\mathbf{w}^T\mathbf{\hat{x}_i}-z_i\|^2 + \gamma\|\mathbf{x_i}-\mathbf{\hat{x}_i}\|_2^2\quad i\in[m].
$$
Then, one can obtain a Stackelberg equilibrium through the classic backward induction procedure \cite{bruckner2011stackelberg}.
With the modified data $\{\x_i^*\}_{i=1}^m$, the learner has to solve the following optimization problem:
\[
\w^* \in \argmin_{\w } \|\w^T \x_i^* - y_i\|^2.
\]
The predictor $\w^*$ and the optimal modified data sets $\{\x_i^*\}_{i=1}^m$ of the data provider are by definition a Stackelberg equilibrium \cite{bruckner2011stackelberg}.
To obtain this, we arrive at the following bi-level optimization problem:
\begin{equation}
\begin{array}{ccl}
& \underset{\mathbf{w}}{\min}\ &\left\|X^{*} \mathbf{w}- \y\right\|^{2} \\
&\ \text{s.t.}\ &X^{*}=\underset{\hat{X}}{\argmin}\  \|\hat{X}{\mathbf{w}}-\mathbf{z}\|^2 +\gamma\|\hat{X}-X\|_{F}^{2},
\end{array}
\label{pb:matrixform}
\end{equation}
where the $i$-th row of ${\bf X} \in\R^{m\times n}$ is just the input example $\x_i$ and the $i$-th entries of $\y, \z\in \R^m$ are labels $y_i$ and $z_i$, respectively.

In their work, \citet{bishop2020optimal} considered the following reformulation. They started by replacing the lower differentiable and strongly convex optimization problem by its optimality condition, i.e.,
$$
X^{*}=\left(\mathbf{z} \mathbf{w}^{T}+\gamma X\right)\left(\mathbf{w} \mathbf{w}^{T}+\gamma I\right)^{-1}.
$$
Then, the Sherman-Morrison formula~\cite{sherman1950adjustment} further implies
$$
X^*\mathbf{w} = \frac{\frac{1}{\gamma} \mathbf{z} \mathbf{w}^{T} \mathbf{w}+X \mathbf{w}}{1+\frac{1}{\gamma} \mathbf{w}^{T} \mathbf{w}}.
$$
Substituting the above formula to problem \eqref{pb:matrixform}, we obtain the following fractional program:
\begin{equation}
\label{pb:ori}
\min_\w~~ \frac{\left\|\frac{1}{\gamma}\z\w^T\w+X\w-\y
-\frac{1}{\gamma}\w^T\w\y\right\|^2}{(1+\frac{1}{\gamma}\w^T\w)^2}.
\end{equation}

\subsection{A Bisection Method for Solving \eqref{pb:ori}}
Here, we briefly review the bisection method developed in \citet{bishop2020optimal} for solving the fractional program \eqref{pb:ori}.

By introduce an artificial variable $\alpha$ and an additional constraint $\alpha = \w^T \w$, \citet{bishop2020optimal} first reformulated \eqref{pb:ori} as the following QFP:
\begin{equation}
\label{pb:frac}
        \begin{array}{ll}
         \min\limits_{\w ,\alpha}& \displaystyle \frac{\|\frac{\alpha}{\gamma}\z+X\w -\y -\frac{\alpha}{\gamma}\y \|^2}{(1+\frac{\alpha}{\gamma})^2} \\
        {\rm s.t.}&\alpha=\w ^T\w.
        \end{array}
\end{equation}
Then, they adopted a bisection search for
$q^*$ such that $F(q^*)=0$, where $F$ is the optimal value function of the following Dinkelbach problem associated with \eqref{pb:frac}, for all $q\in \R$,
\begin{equation}
\label{pb:bisub}
\begin{aligned}
 F(q):=\min\limits_{\w ,\alpha}\quad& \|\frac{\alpha}{\gamma}\z+X\w -\y -\frac{\alpha}{\gamma}\y \|^2-q(1+\frac{\alpha}{\gamma})^2\\
 \rm s.t \quad &\alpha=\w ^T\w.
\end{aligned}
\end{equation}
The correctness of their algorithm is due to the following well known result for fractional programming.
\begin{lemma}
[Theorem 1 of \citet{dinkelbach1967nonlinear}]
\label{lem:fq}
Assume that for all $q\in\R$, problem \eqref{pb:bisub} has nonempty optimal solution set. Then, the equation $F(q)=0$ has a unique solution. Furthermore, $(\w^*,\alpha^*)$ is a solution to the QFP \eqref{pb:frac} if and only if $\w^{*T} \w^* = \alpha^*$ and $F(q^*)=0$ where $q^* = \|\frac{\alpha^*}{\gamma}\z+X\w^* -\y -\frac{\alpha^*}{\gamma}\y \|^2/(1 + \alpha^*/\gamma)^2$.
\end{lemma}
As $F(q)$ is a concave monotonically decreasing continuous function \cite{dinkelbach1967nonlinear}, the bisection algorithm is well-defined. \citet{bishop2020optimal} further showed that initial lower and upper bounds $q_1$ and $q_2$ for $q^*$ satisfying $F(q_1)\ge0$ and $F(q_2)\le0$ are also easy to obtain.

In each iteration of the bisection method, given $q$, one needs to compute $F(q)$, i.e., the nonconvex optimization problem \eqref{pb:bisub} needs to be solved. To this purpose, \citet{bishop2020optimal} applied the S-lemma with equality \cite{xia2016s} to transform the QCQP \eqref{pb:bisub} into an SDP problem whose optimal objective is exactly $F(q)$. More specifically, define matrices
\begin{equation*}
\begin{array}{l}
\hat A=  \begin{smallpmatrix}
        X^T X&\frac{1}{\gamma} X^{T}(\mathbf{z}-\mathbf{y}) & -X^T \mathbf{y}\\
        \frac{1}{\gamma}(\mathbf{z}-\mathbf{y})^{T}X&\frac{1}{\gamma^2}\|\mathbf{z}-\mathbf{y}\|^2&-\frac{1}{\gamma}(\mathbf{z}-\mathbf{y})^T \mathbf{y}\\
        -\mathbf{y}^T X & -\frac{1}{\gamma}\mathbf{y}^T (\mathbf{z}-\mathbf{y})&\mathbf{y}^T \mathbf{y}\\
        \end{smallpmatrix},\\
 \hat B=
          \begin{smallpmatrix}
          \bm{0}_{n}& & \\
           &\frac{1}{\gamma^2}&\frac{1}{\gamma}\\
           &\frac{1}{\gamma}&1\\
          \end{smallpmatrix}\text{ and }
\hat C =\begin{smallpmatrix}
          {I}_n& &\\
          & 0 & -\frac{1}{2}\\
          & -\frac{1}{2}&0\\
          \end{smallpmatrix}.
\end{array}
\end{equation*}
Given $q\in\R$, problem \eqref{pb:bisub} admits the same objective value with the following SDP
\begin{equation}
\label{pb:SDPb}
\max_{\tau,\lambda}~\tau \quad{\rm s.t.}~\hat A - q\hat B + \lambda\hat C - \tau  E\succeq0,
\end{equation}
where $E=\Diag(0_{n+1},1) \in \R^{(n+2)\times (n+2)}$ is the diagonal matrix with first $n+1$ diagonal entries being zero and the last entry being one. Then, the SDP \eqref{pb:SDPb} is solved by advanced interior point methods (IPM). 

Theoretically, \citet{bishop2020optimal}
showed that under the assumption that each involved SDP is solved exactly, the bisection method needs $\log_2(2 \y^T \y/\varepsilon)$ steps to obtain an $\varepsilon$-optimal estimation of $q^*$ with given tolerance $\varepsilon >0$. Note that in practice,
$\y^T \y$ can be quite large and $\varepsilon$ may be required to be small. Thus, the bisection method may need to solve a significant numbers of SDPs even in the moderate-scale setting, e.g., the numbers of samples $m$ and features $n$ are several thousands. Since the amount of work per iteration of IPM for solving \eqref{pb:SDPb} is ${\cal O}(n^3)$ \cite{Nesterov1993interior, Todd2001semidefinte}, the total computational costs of the bisection method can be prohibitive. Moreover, there in fact exists no optimization solver which can return exact solutions to these SDPs. Hence, the convergence theory of the bisection method may break down and its stability may be implicitly affected due to the accumulation of optimization errors in each iteration. These issues on scalability and stability of the bisection method motivate our study in this paper.

\section{Single SDP Reformulation}
In this section, we present a novel result that shows an optimal solution to \eqref{pb:ori}, or the Stackelberg equilibrium of the SPG-LS, can be obtained by just solving a single SDP with a similar size as the SDP \eqref{pb:SDPb}.
To begin, let us consider the following equivalent formulation of \eqref{pb:ori}:
\begin{equation}
\label{pb:ourfrac}
\begin{array}{cccc}
&  \underset{\mathbf{w},\ \alpha}{\min}\ &\displaystyle \frac{\left\|\alpha \mathbf{z}+X \mathbf{w}-\mathbf{y}-\alpha \mathbf{y}\right\|^{2}}{\left(1+\alpha\right)^{2}}\quad &\text{s.t.}\quad \frac{\mathbf{w}^{T}\mathbf{w}}{\gamma}=\alpha,
\end{array}
\end{equation}
which is slightly different from \eqref{pb:frac} in a scaling of $\alpha$.
Now let us recall the following \textit{S-lemma with equality}, which is the main tool in showing the equivalence of \eqref{pb:bisub} and \eqref{pb:SDPb} in \citet{bishop2020optimal}.
\begin{lemma}[Theorem 3 in \citet{xia2016s}]
        \label{lem:S}
        Let $f(\x)=\x^T Q_1 \x+2\p_1^T \x+q_1$ and $h(\x)= \x^T Q_2\x+2\p_2^T \x+q_2$, where $Q_1,Q_2\in\R^{n\times n}$ are symmetric matrices, $\p_1, \p_2\in\R^n$ and $q_1, q_2\in \R$. If function $h$ takes both positive and negative values and $Q_2\neq 0$, then following two statements are equivalent:
        \begin{enumerate}
                \item There is no $\x\in \R^n$ such that $f(\x) <0$, $h(\x)=0$.
                \item There exists a $\lambda\in\mathbb{R}$ such that $f(\x)+\lambda h(\x)\ge0$.
        \end{enumerate}
\end{lemma}
We also need the following result that is well known in quadratic programming.
\begin{lemma}
        [Theorem 2.43 in \citet{beck2014introduction}]
        \label{lem:nonnegqp}
        Let $Q\in\R^{n\times n}$ be a symmetric matrix, $\p\in\R^n$ and $q\in\R$. Then the following two statements are equivalent:
        \begin{enumerate}
                \item $(\x^T, 1) \begin{pmatrix}Q &\p\\ \p^T&q\end{pmatrix}\begin{pmatrix}\x \\ 1\end{pmatrix} \ge 0$ for all $\x\in\R^n$.
                \item $\begin{pmatrix}Q &\p\\ \p^T&q\end{pmatrix}\succeq0$.
        \end{enumerate}
\end{lemma}
From now on, let us define
\begin{equation}
\begin{array}{l}
\label{eq:formABC}
A=  \begin{smallpmatrix}
        X^T X& X^{T}(\mathbf{z}-\mathbf{y}) & -X^T \mathbf{y}\\
        (\mathbf{z}-\mathbf{y})^{T}X&\|\mathbf{z}-\mathbf{y}\|^2&-(\mathbf{z}-\mathbf{y})^T \mathbf{y}\\
        -\mathbf{y}^T X & -\mathbf{y}^T (\mathbf{z}-\mathbf{y})&\mathbf{y}^T \mathbf{y}\\
        \end{smallpmatrix},\\
\displaystyle B=
          \begin{smallpmatrix}
          \bm{0}_{n}& & \\
           &1&1\\
           &1&1\\
          \end{smallpmatrix}\text{ and }
C =\begin{smallpmatrix}
          \frac{{I}_n}{\gamma}& &\\
          & 0 & -\frac{1}{2}\\
          & -\frac{1}{2}&0\\
          \end{smallpmatrix}.
\end{array}
\end{equation}
With the above facts, we are now ready to present our main result of this section that  \eqref{pb:ori} can be equivalently reformulated as a single SDP, where  our
SDP reformulation follows a similar idea in equations (1.12-1.14) in \cite{nguyen2014sdp}.
\begin{theorem}
        Problem \eqref{pb:ori} is equivalent to the following SDP
        \begin{align}
        \label{pb:SDP}
        \begin{array}{lll}
        &\sup_{\mu,\lambda}& \mu\\
        &\rm s.t.&A-\mu B+\lambda C\succeq0.
        \end{array}
        \end{align}
\end{theorem}
\begin{proof}
        Consider the equivalent formulation \eqref{pb:ourfrac}.
        Let $f(\w,\alpha) = \left\|\alpha \mathbf{z}+X \mathbf{w}-\mathbf{y}-\alpha \mathbf{y}\right\|^{2}$, $p(\w,\alpha) = \frac{\mathbf{w}^{T}\mathbf{w}}{\gamma}-\alpha $ and denote by $v_{\rm frac}$ the optimal value of \eqref{pb:ourfrac}.
        Recall the definitions of $A,B$, and $C$ in \eqref{eq:formABC}.
        Then, we conduct the reformulation in the following manner:
        \begin{align}
        v_{\rm frac} & = \underset{\w,\alpha}{\inf}~  \left\{\frac{f(\w,\alpha)}{(1+\alpha)^2}: p(\w,\alpha)=0 \right\} \nonumber \\
        & = \sup_\mu \left\{\mu:  \begin{array}{ll}
        \{(\w,\alpha)\mid f(\w,\alpha)-\mu(1+\alpha)^2 < 0,\\
        p(\w,\alpha)=0\}=\emptyset\end{array} \right\} \nonumber\\
        & = \sup_\mu\left\{\mu: \begin{array}{l}
        \exists\lambda\in\R \text{ s.t.} f(\w,\alpha)-\mu(1+\alpha)^2\\
        +\lambda p(\w,\alpha)\ge0,\ \forall \w\in \R^n,\alpha\in\R \end{array}\right\}\label{pb:S2}\\
        \ & = \sup_{\mu,\lambda} \left\{\mu:
        \begin{array}{l}  \begin{smallpmatrix}
        \w^T&  \alpha &         1
        \end{smallpmatrix}(A-\mu B+\lambda C)\begin{smallpmatrix}
        \w \\
        \alpha \nonumber\\
        1
        \end{smallpmatrix}\ge0,\\
        \forall \w\in \mathbb{R}^n,\alpha\in\mathbb{R}
        \end{array}\right\} \nonumber\\
        \ & = \sup_{\mu,\lambda} \left\{\mu:  A-\mu B+\lambda C\succeq0\right\}\label{pb:singlesdp},
        \end{align}
        where \eqref{pb:S2} is due to the S-lemma with equality in Lemma \ref{lem:S} and \eqref{pb:singlesdp} is due to Lemma \ref{lem:nonnegqp}.
\end{proof}

We briefly remark that there exists an optimal solution for the SDP \eqref{pb:SDP} and it can be used to recover an optimal solution to  \eqref{pb:ori}. In fact, we can recover an optimal solution to \eqref{pb:ori} by either doing a rank-1 decomposition for the dual solution of SDP \eqref{pb:SDP}, thanks to \citet{sturm2003cones},  or solving a linear system with an additional equation as in step 8 in Algorithm \ref{alg:socp}\footnote{See the discussions after Theorem \ref{thm:socp}.}.
More details are given in Appendix.

Up to now, we have shown that to obtain a global optimal solution to the nonconvex fractional program \eqref{pb:ori}, only a single SDP needs to be solved.
A crucial observation is that our single SDP \eqref{pb:SDP} has a similar form and the same dimension of the matrices with \eqref{pb:SDPb}, the subproblem in each iteration of the bisection method.
We remark the main differences: (i) the bisection parameter $q$ is the variable $\mu$ in our formulation; (ii) our formulation does not involve a $\tau$ which is used for generating new half interval in the bisection method.
From the similar forms of two SDPs, we can expect that solving the SDPs \eqref{pb:SDP} and \eqref{pb:SDPb} needs a similar CPU time.
However, the bisection method needs to solve a series of SDPs.
Indeed, for each test instance in our numerical experiments, the bisection method needs to solve about 30 SDPs.
In other words, our single SDP method is a more efficient way to obtain $q^*$ (or equivalently, $\mu$ in \eqref{pb:SDP}) such that $F(q^*)=0$, which closely relates to an optimal solution of problem \eqref{pb:frac} (or equivalently, problem \eqref{pb:ori}) in view of Lemma \ref{lem:fq}\footnote{In fact, we use Lemma 2.1 slightly different from its original statement with a scaling of $\gamma$ here and in the discussions after Theorem \ref{thm:socp}.},  than the bisection method.

\section{SOCP Reformulation}
\label{sec:SOCP}
Though our single SDP approach introduced in the previous section for finding Stackelberg equilibrium of SPG-LS has already been much faster than the bisection method, the fact that solving a large-scale SDP requires extensive computations motivates us to make a step further of seeking more reductions.
For this purpose, in this section, by using a simultaneous diagonalizability of submatrices in the linear matrix inequality (LMI) constraint in \eqref{pb:SDP}, we can further reformulate SDP \eqref{pb:SDP} as an SOCP that can be solved much more efficiently. We briefly describe our main idea in Algorithm \ref{alg:socp}.

\begin{algorithm}[hptb]
\caption{SOCP method for solving \eqref{pb:ori}}
\label{alg:socp}
\begin{algorithmic}[1]
\STATE {\bfseries Input:} matrices $A,B,C$ in \eqref{eq:formABC}
\STATE set $V_1$ as in \eqref{eq:v1}
\STATE set $\bar A,\bar B,\bar C$ as in \eqref{eq:formbarA}, \eqref{eq:formbarB}, \eqref{eq:formbarC}
\STATE do spectral decomposition to matrix $\bar A_{11}$ in \eqref{eq:formbarA} with $\bar{A}_{11} = HDH^{T}$
\STATE set $V_2$ as in \eqref{eq:V2}
\STATE obtain the matrices $\tilde A=V_2\bar AV_2,\tilde B=V_2\bar BV_2,\tilde C=V_2\bar CV_2$ in forms \eqref{eq:formtilA} and \eqref{eq:formtilBC} with diagonal $n+1$th order leading principal submatrices
\STATE solve the SOCP problem \eqref{pb:SOCP} to obtain optimal $\mu^*,\lambda^*$
\STATE obtain $\w^*$ by finding a solution of the following linear system
\begin{equation*}
(A-\mu^* B+\lambda^* C)\begin{smallpmatrix}
        \w \\
        \alpha \nonumber\\
        1
        \end{smallpmatrix}=0
\end{equation*}
satisfying $\frac{1}{\gamma}\w^T\w=\alpha$
\end{algorithmic}
\end{algorithm}

The motivation of our reformulation comes from simple observations on matrices $A$, $B$ and $C$.
The first key observation is that $B$ and $C$ can be simultaneously diagonalized by congruence. Indeed, letting
\begin{equation}
\label{eq:v1}
V_1=
     \begin{smallpmatrix}
     I_n&0&0\\
     0&\frac{1}{\sqrt{\gamma}}&1\\
     0&-\frac{1}{\sqrt{\gamma}}&1\\
     \end{smallpmatrix},
\end{equation}
 we have from \eqref{eq:formABC}
\begin{equation}
\label{eq:formbarB}
\bar B:=V_1^TBV_1=\begin{pmatrix}
\bm {0}_{n+1}&\\
&4
\end{pmatrix},
\end{equation}  and
\begin{equation}
\label{eq:formbarC}
\bar C:=V_1^TCV_1= \begin{pmatrix}
\frac{1}{\gamma}I_{n+1}&\\
&-1
\end{pmatrix}.
\end{equation}
For convenience, let
\begin{equation}
\label{eq:formbarA}
\bar A:=V_1^TAV_1 =\begin{pmatrix}
          \bar A_{11}&\bar A_{12}\\
          \bar A_{12}^T&
          \bar A_{22}
          \end{pmatrix}.
\end{equation}
The second key observation is that
the $n+1$th order leading principal submatrices of $A,B$ and $C$ can be  simultaneously diagonlizable by congruence.
To see this, applying spectral decomposition to the real symmetric matrix $\bar A_{11}$ yields $\bar{A}_{11} = HDH^{T}$, where $H$ is an $(n+1)\times(n+1)$ orthogonal matrix and 
 $D=\Diag(\bd)$
is a diagonal matrix with $d_i$ being its $i$th diagonal entry.
Define
\begin{equation}
\label{eq:V2}
V_2= \begin{pmatrix}
     H&0\\
     0&1\\
     \end{pmatrix}.
\end{equation}
Now we have
\begin{equation}
\label{eq:formtilA}
\tilde A:= V_2^T\bar AV_2 =
\begin{pmatrix}
D & \bb\\
\bb^T&c\\
\end{pmatrix},
\end{equation}
where $\bb\in\R^{n+1}$ and $c\in\R$.
Since $H^TH=I$, we also have
\begin{equation}
\label{eq:formtilBC}\tilde B:= V_2^T\bar BV_2 =\bar B\text{~~ and ~~} \tilde C= V_2^T\bar CV_2=\bar C.
\end{equation}
As $V_1$ and $V_2$ are both invertible matrices, the LMI constraint in \eqref{pb:SDP} is equivalent to
\begin{equation}
\label{eq:tildeLMI}
\tilde A-\mu\tilde B+\lambda \tilde C\succeq 0.
\end{equation}
From the generalized Schur complement \cite{zhang2006schur}, the LMI \eqref{eq:tildeLMI} is equivalent to
\begin{equation}
\label{eq:Schur}
\begin{array}{l}
D+\tfrac{\lambda}{\gamma}I_{n+1} \succeq 0,\\
\bb\in {\rm Range}(D+\tfrac{\lambda}{\gamma}I_{n+1}),\\
c-4\mu-\lambda-\bb^{T}(D+\tfrac{\lambda}{\gamma}I_{n+1} )^{\dagger}\bb \succeq  0,
\end{array}
\end{equation}
where $(M)^{\dagger}$ denotes the Moore-Penrose pseudoinverse of matrix $M$.
As $D$ is a diagonal matrix, by defining $\frac{0}{0}=0$, \eqref{eq:Schur} is further equivalent to
\begin{equation*}
\begin{array}{l}d_i+\tfrac{\lambda}{\gamma}\ge0, \mbox{ and } b_i=0\text{ if }d_i+\tfrac{\lambda}{\gamma}=0, ~i\in [n+1],\\
c -4\mu-\lambda- \sum_{i=1}^{n+1}\frac{b_i^2}{d_i+\lambda/\gamma}\ge0.
\end{array}
\end{equation*}
These constraints can be further rewritten as
\begin{equation*}
\begin{array}{l}
d_i+\tfrac{\lambda}{\gamma}\ge0, ~ i\in [n+1], \\
c-4\mu-\lambda- \sum_{i=1}^{n+1} s_i\ge0,\\
s_i(d_i+\tfrac{\lambda}{\gamma})\ge b_i^2,~ i \in [n+1].
\end{array}
\end{equation*}
For $i\in [n+1]$, each constraint  $s_i(d_i+\tfrac{\lambda}{\gamma})\ge b_i^2$
can be expressed as a rotated second order cone constraint \cite{alizadeh2003second}, which is equivalent to the second order cone constraint
\[
\sqrt{b_i^2 +\left(\tfrac{s_i+d_i-\tfrac{\lambda}{\gamma}}{2}\right)^2 }\le \frac{s_i+d_i+\tfrac{\lambda}{\gamma}}{2}
.
\]

Consequently, we have the following theorem.
\begin{theorem}
\label{thm:socp}
With the same notation in this section, problem \eqref{pb:SDP} is  equivalent to the following SOCP problem
\begin{align}
\label{pb:SOCP}
    \begin{array}{ll}
         \sup_{\mu,\lambda,\s}& \mu\\
         \rm s.t.&d_i+\tfrac{\lambda}{\gamma}\ge0, ~ i\in [n+1], \\  &c-4\mu-\lambda-\sum_{i=1}^{n+1}s_i\ge0,\\   &s_i(d_i+\tfrac{\lambda}{\gamma})\ge~ b_i^2,~i \in [n+1].
    \end{array}
\end{align}
\end{theorem}

Note that based on our construction, i.e., the congruence transformations to the matrices in the LMI constraint in \eqref{pb:SDP},  any optimal solution $(\mu^*,\lambda^*)$  to \eqref{pb:SOCP} is still optimal to \eqref{pb:SDP}.
We claim {that} an optimal {solution} $\w^*$ to \eqref{pb:ori} can be recovered by solving the linear system with an additional equation in step 8 in Algorithm \ref{alg:socp}. Indeed, Lemma \ref{lem:fq}, the strong duality theory of SDPs and the S-lemma with equality guarantee the existence of a rank-1 solution to the SDP relaxation of
\begin{equation*}
\begin{aligned}
\min\limits_{\w ,\alpha}\quad& \left\|\alpha \mathbf{z}+X \mathbf{w}-\mathbf{y}-\alpha \mathbf{y}\right\|^{2} -\mu^*(1+\frac{\alpha}{\gamma})^2\\
\rm s.t \quad &\frac{1}{\gamma}\w ^T\w=\alpha,
\end{aligned}
\end{equation*}
and the solution solves the following KKT system of the corresponding SDP\ relaxation
\begin{equation*}
\left\{
\begin{array}{l}
\langle C,W\rangle=0,\\
W_{n+2,n+2}=1,\\
W\succeq 0, \\
\langle A-\mu^{*} B+\lambda^* C,W \rangle  =0.
\end{array}
\right.
\end{equation*}
By setting $W=\begin{smallpmatrix}
        \w \\
        \alpha \nonumber\\
        1
        \end{smallpmatrix}\begin{smallpmatrix}
        \w^T &
        \alpha \nonumber&
        1
        \end{smallpmatrix}$, the above facts are equivalent to
\begin{equation*}
(A-\mu^* B+\lambda^* C)\begin{smallpmatrix}
        \w \\
        \alpha \nonumber\\
        1
        \end{smallpmatrix}=0,\quad \frac{1}{\gamma}\w^T\w=\alpha,
\end{equation*}
due to $A-\mu^{*} B+\lambda^* C\succeq0$.
One may think that the above equations are difficult to solve. In fact, the linear system usually only has a unique solution and it suffices to solve the linear system solely. A sufficient condition to guarantee this is that the matrix $(A-\mu^{*} B+\lambda^{*} C)$ is of rank $n+1$, which is exactly the case in all our numerical tests.
More discussions on the solution recovering are given in Appendix.

In general, SOCPs can be solved much faster than SDPs. For our problem, it can be seen that IPMs for solving SOCP \eqref{pb:SOCP} takes ${\cal O}(n)$ costs per iteration \cite{alizadeh2003second, Andersen2003implementing,Tutuncu2003solving} which is of orders magnitudes faster than the case ${\cal O}(n^3)$ in  solving SDP \eqref{pb:SDP} using interior point methods. The high efficiency of our SOCP approach is also evidenced by our numerical tests.
\section{Experiment Results}
In this section, we conduct numerical experiments on both synthetic and real world datasets to verify the superior performance of our proposed algorithms in terms of both the computational time and the learning accuracy.  We apply the powerful commercial solver MOSEK  \cite{aps2021mosek}  to solve all the SDPs and SOCPs  in the bisect method and ours. 


All simulations are implemented using MATLAB R2019a on a PC running Windows 10 Intel(R) Xeon(R) E5-2650 v4 CPU (2.2GHz) and 64GB RAM. We report the results of two real datasets and three synthetic datasets and defer other results to the supplementary material.\footnote{Our code is available at \url{https://github.com/JialiWang12/SPGLS}.}
\subsection{Real World Dataset}\label{sec:5.1}
We first demonstrate the accuracy and efficiency of our proposed methods on two real datasets. We compare the average mean squared error (MSE) as well as the wall-clock time of our SDP and SOCP approaches with those of the bisection method in \citet{bishop2020optimal}, the ridge regression and a nonlinear programming reformulation of the SPG-LS in \citet{bruckner2011stackelberg}. Similar as in \citet{bishop2020optimal}, to evaluate the learning accuracy of the algorithms, we perform 10-fold cross-validation and compare their average MSE for 40 different values of the parameter  $\gamma\in[1\times10^{-3},0.75]$ in \eqref{pb:ori}.  For each   $\gamma$, a grid search on 9 logarithmically spaced points $[1\times 10^{-5},1000]$ is used to compute the best regularization
parameter for the ridge regression.   We also compare the running time of all the methods  at $\gamma = 0.5$, averaged over 10 trials to further illustrate the efficiency of our methods.
For the testing purpose, we first apply min-max normalization to the raw data $X$ and scale the labels $y$, $z$ to $y = y/(\beta\|y\|_\infty)$ and $z = z/(\beta\|y\|_\infty)$, respectively. These labels will be scaled back to compute  the average MSE. It is worth noting that the constant $\beta$ can be adjusted with respect to different datasets.
\subsubsection{Wine Dataset}
\begin{figure*}[htbp]
\centering
\subfigure{
\includegraphics[scale=0.27]{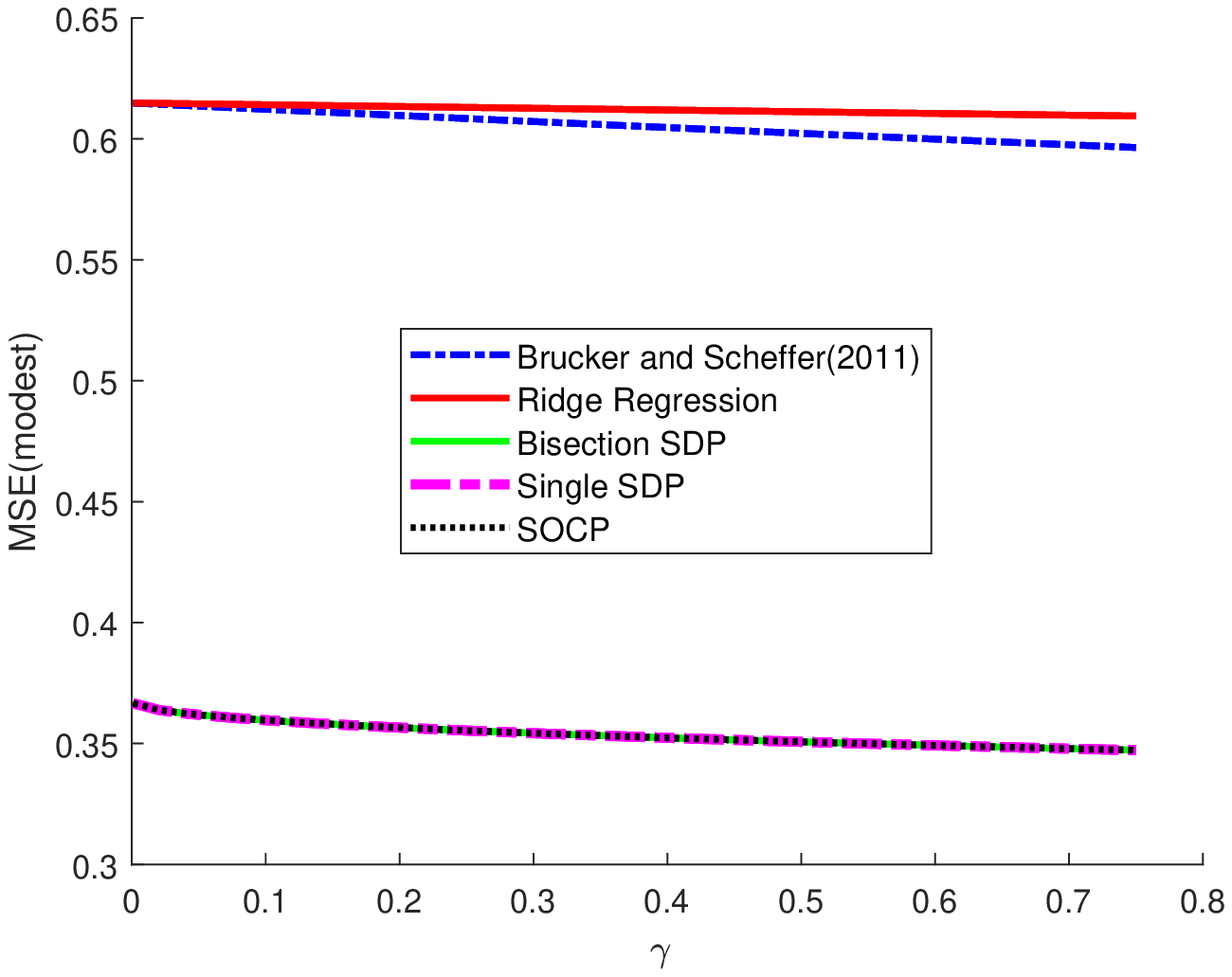} \label{fig:winemodestmse}
}
\subfigure{
\includegraphics[scale=0.27]{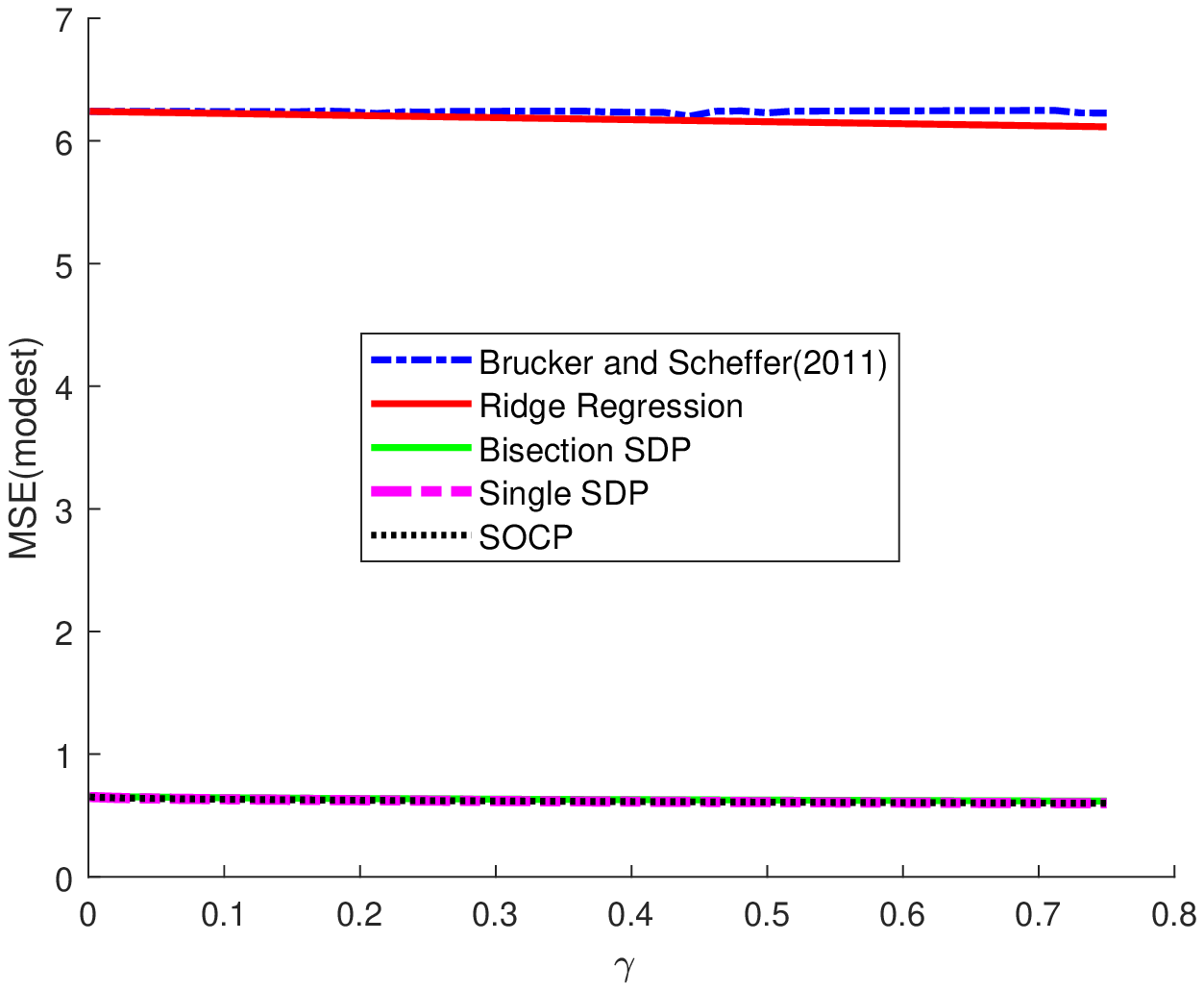} \label{fig:wineseveremse}
}
\subfigure{
\includegraphics[scale=0.27]{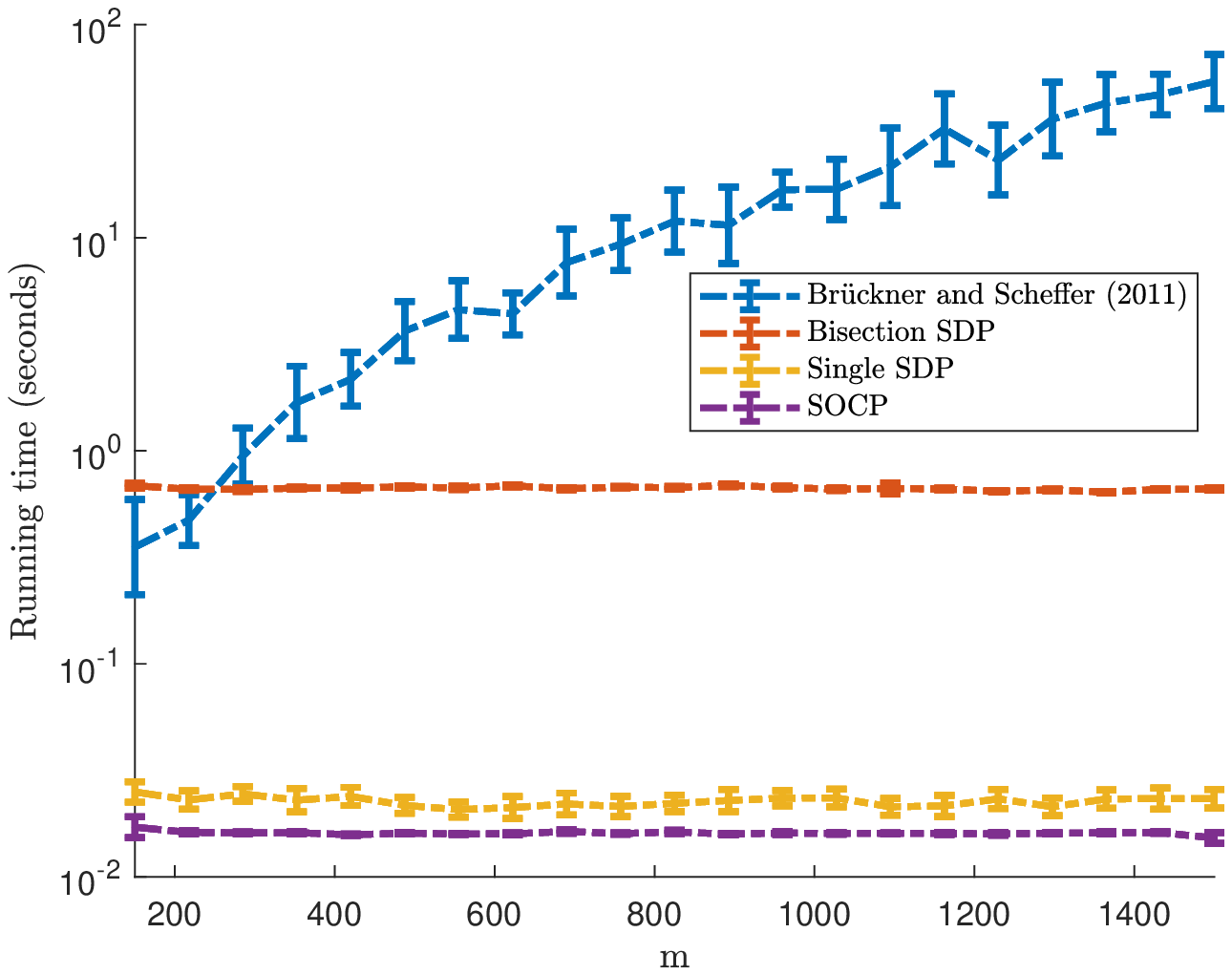} \label{fig:winemodestime}
}
\subfigure{
\includegraphics[scale=0.27]{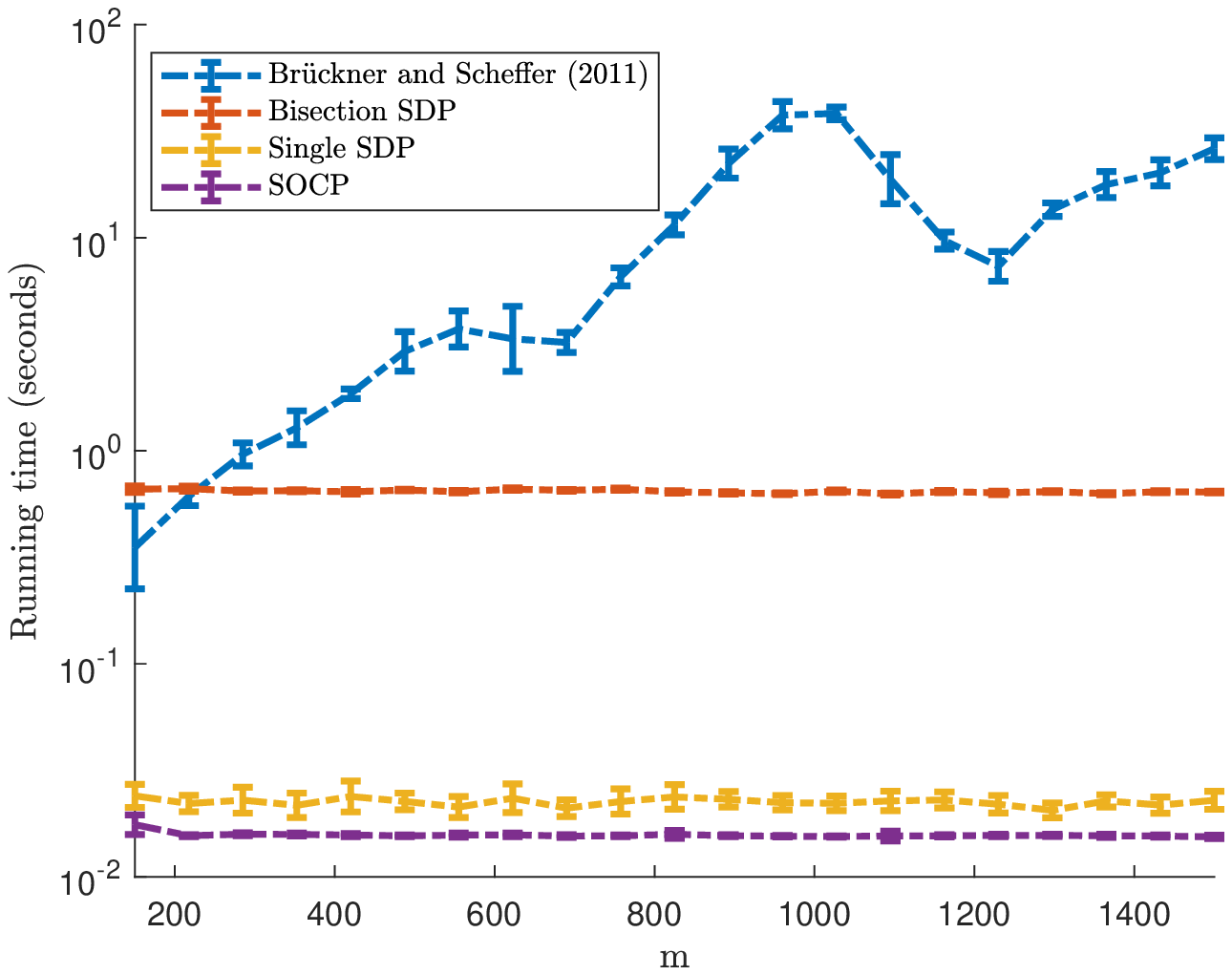} \label{fig:wineseveretime}}
\caption{Performance comparison between different algorithms on the red wine dataset.  The left two plots correspond to MSE result generated  by $\mathcal A_{\rm modest}$ and $\mathcal A_{\rm severe}$, whilst the right two plots correspond to wall-clock time comparison generated  by $\mathcal A_{\rm modest}$ and $\mathcal A_{\rm severe}$.}
\label{fig:winemsetime}
\end{figure*}

\begin{figure*}[htbp]
\centering
\subfigure{
\includegraphics[scale=0.27]{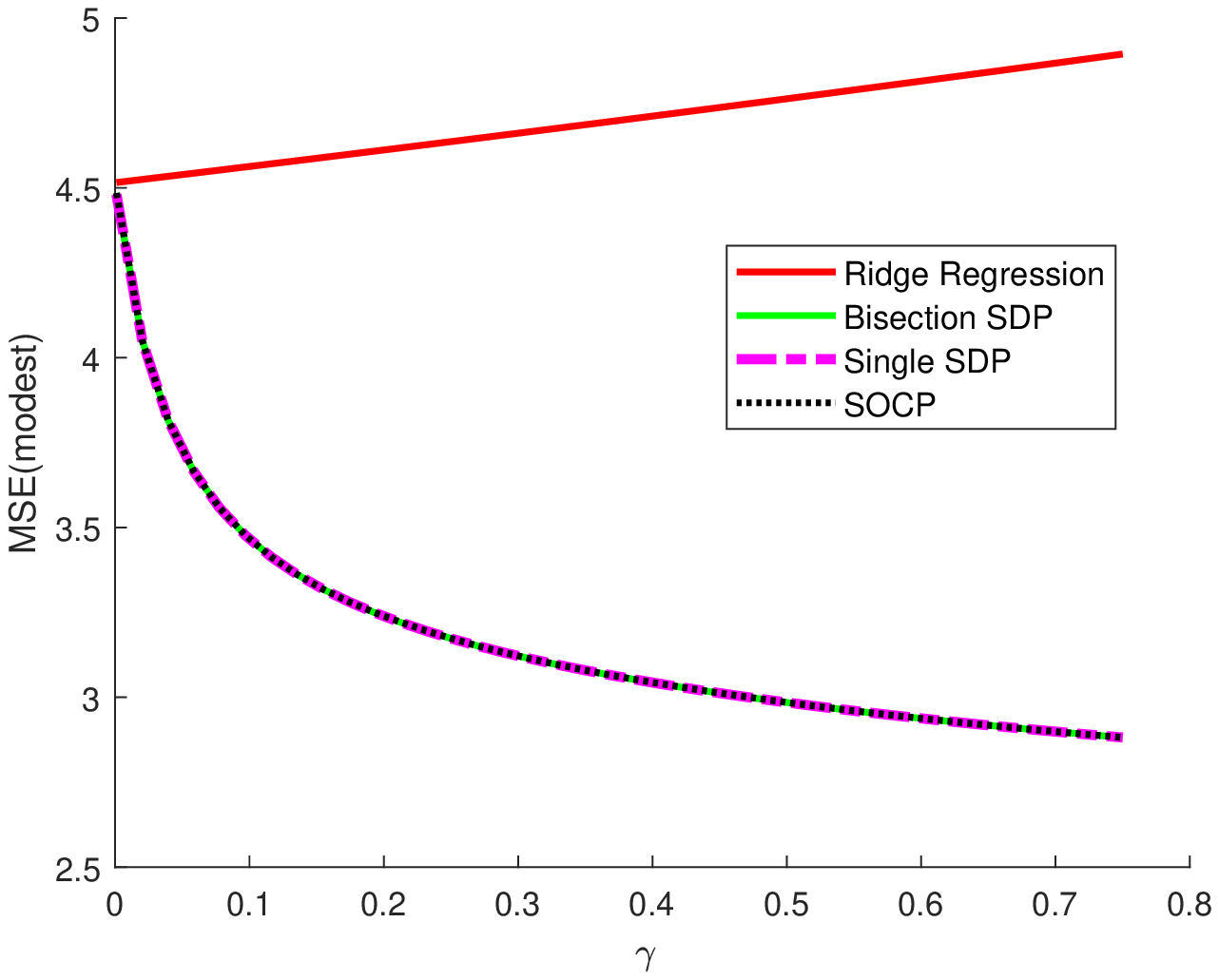} \label{fig:blogmodestmse}
}
\subfigure{
\includegraphics[scale=0.27]{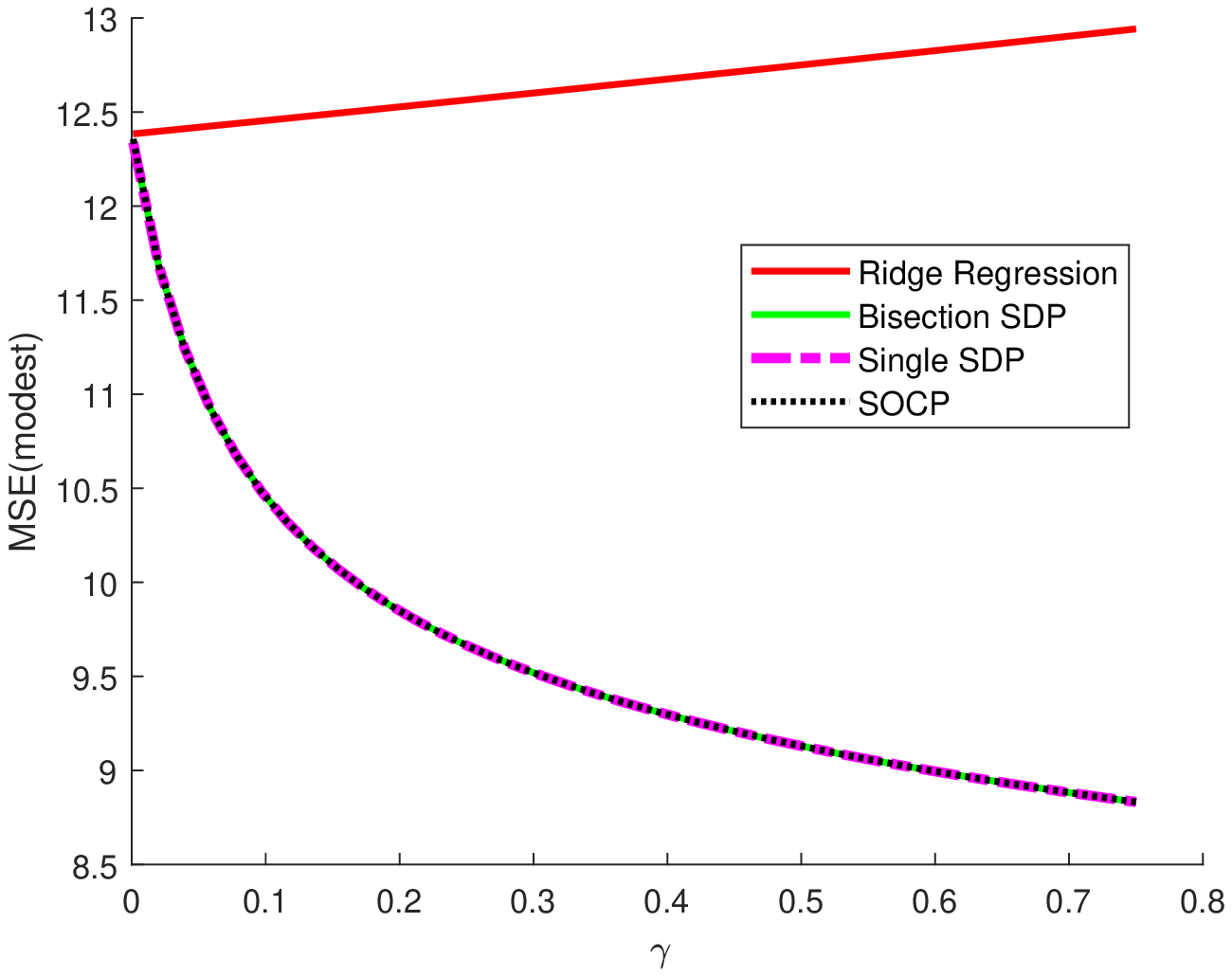} \label{fig:blogseveremse}
}
\subfigure{
\includegraphics[scale=0.27]{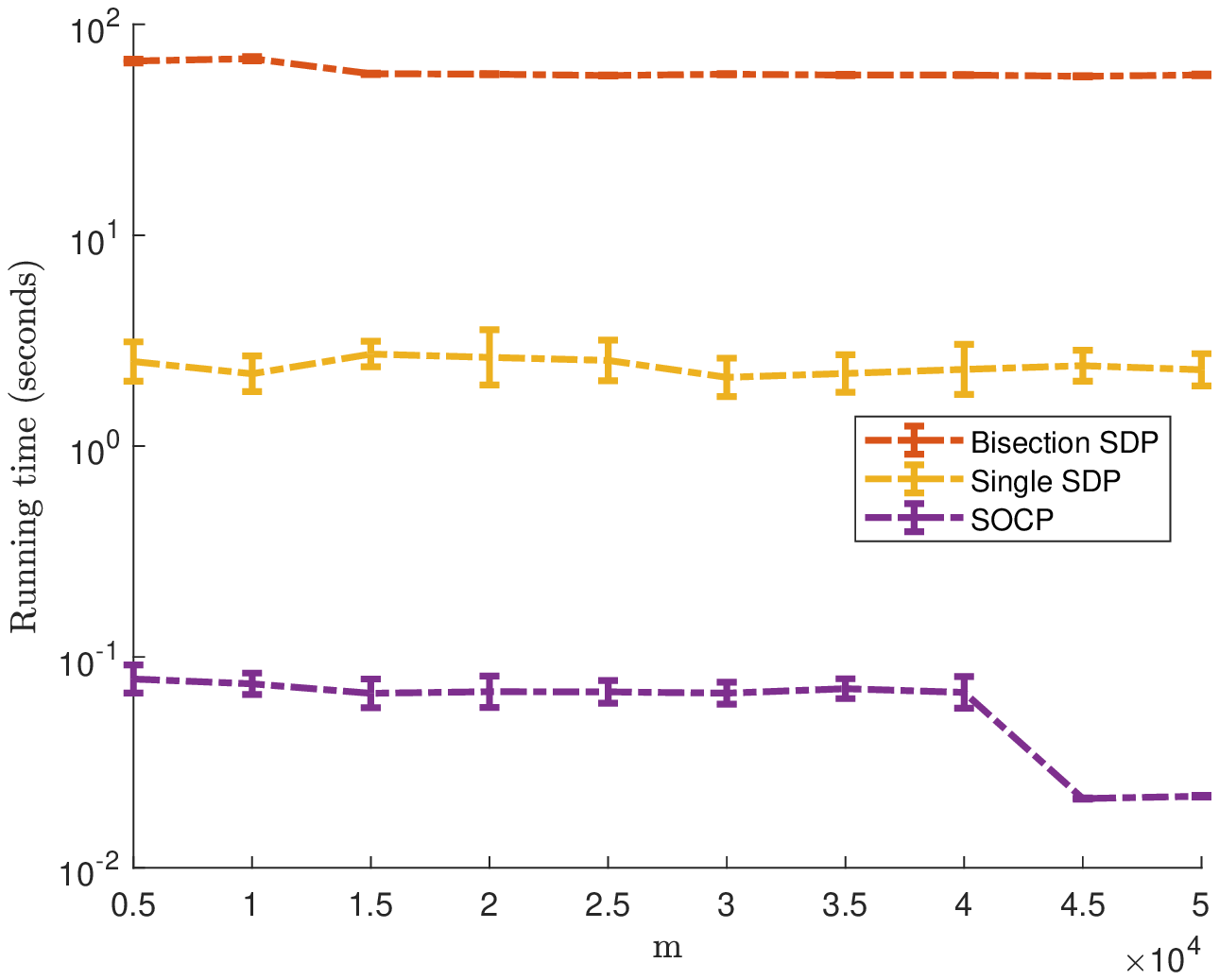} \label{fig:blogmodestime}
}
\subfigure{
\includegraphics[scale=0.27]{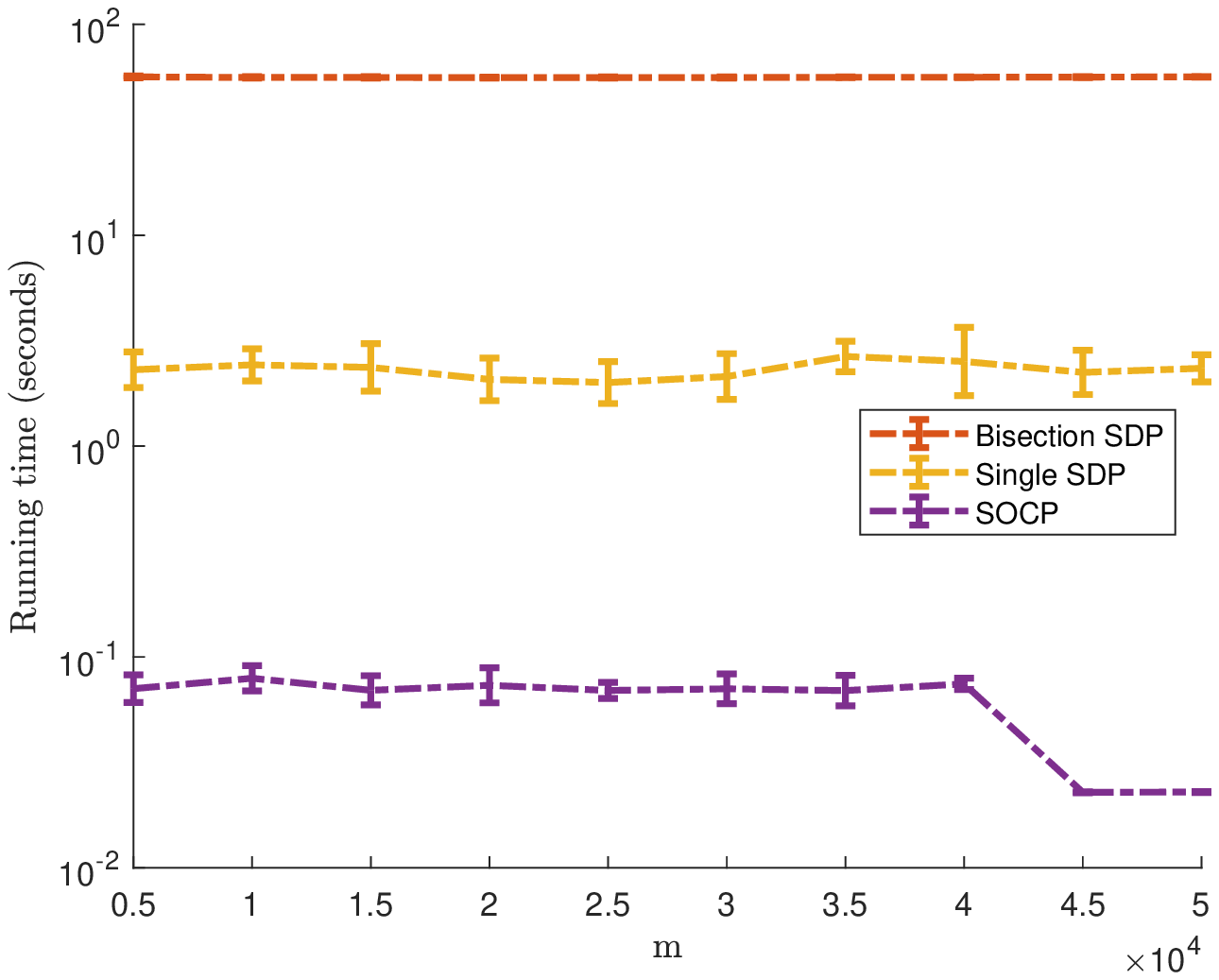} \label{fig:blogseveretime}}
\caption{Performance comparison between different algorithms on the blog dataset.  The left two plots correspond to MSE result generated  by $\mathcal A_{\rm modest}$ and $\mathcal A_{\rm severe}$, whilst the right two plots correspond to wall-clock time comparison generated  by $\mathcal A_{\rm modest}$ and $\mathcal A_{\rm severe}$.}
\label{fig:blogmsetime}
\end{figure*}



We first test our methods on the red wine dataset \cite{CORTEZ2009547}, which contains 1599 instances each with 11 features. The response is a physiochemical measurement ranged from 0 to 10, where higher score means better quality. We use the same setting as in \citet{bishop2020optimal}. The wine provider manipulates the data to achieve a higher score if the original label is smaller than some threshold $t$. The wine provider sets his target label $z$ as follows,
\begin{equation*}
z_i = \max\{y_i, t\}\label{winez}.
\end{equation*}
We consider two different providers  $\mathcal A_{\rm modest}$ with $t_{\rm modest}= 6$ and  $\mathcal A_{\rm severe}$ with  $t_{\rm severe}= 8$.

Our numerical results are reported in Figure \ref{fig:winemsetime}. From Figures \ref{fig:winemodestmse} and \ref{fig:wineseveremse}, we see that our single SDP method, our SOCP method and the bisection method achieved the best performance in  average MSE. This is not surprise as the three methods are guaranteed to solve the SPG-LS globally. Figures \ref{fig:winemodestime} and \ref{fig:wineseveretime} indicate that both single SDP and SOCP are much faster than all the other three methods. Since the dimension of SDP is rather small, our SDP method took a similar time with our SOCP method.

\subsubsection{Blog Dataset}
We next compare our algorithms on the blogfeedback dataset\footnote{https://archive.ics.uci.edu/ml/datasets/BlogFeedback} from the UCI data repository \cite{asuncion2007uci}.
It consists of 52397 data processed from raw feedback-materials collected from the Internet. Each one conveys the information of a certain session, described by 281 features. The response is the number of comments. The task for the learner, in this case, is to predict the the future comment numbers in a regression manner.


As before, we assume that the label $ z_i =\max\{ y_i + \delta,0\} $ is modified by the data provider in order to trigger a biased result. For example, consider an option guider who aims to manipulate the public expectation of a certain blog news. He is then motivated to temper the announced comment number. We assume there are two types of data providers, $\mathcal A_{\rm modest}$ with  $ \delta =-5$  and  $\mathcal A_{\rm severe}$ with $ \delta=-10 $. 
All the other hyperparameters are  the same with the wine dataset.

For this dataset,  we do not compare the  method in \citet{bruckner2011stackelberg} for time consideration. Hence we only present the comparisons of the other four methods
in  Figure \ref{fig:blogmsetime}.  Similarly, Figures \ref{fig:blogmodestmse} and \ref{fig:blogseveremse} demonstrate that our two methods achieved the best average MSE. Figures \ref{fig:blogmodestime} and \ref{fig:blogseveretime} indicate that both the single SDP and SOCP methods are much faster than the bisection method, and the SOCP approach surpasses the single SDP approach. In fact, our SOCP method takes only  about $1/50$ time of our single SDP method, while our single SDP method takes only  about $1/20$ time of the bisection method. That is, our SOCP method is 1000 times more efficient than the bisection method on this dataset.

\subsection{Synthetic Dataset}\label{sec:5.2}
To further demonstrate the efficacy of our proposed approaches in large-scale problems, we perform synthetic experiments with a high feature dimension.  The function \texttt{make\_regression} in scikit-learn~\cite{pedregosa2011scikit} is used to build artificial datasets of controlled size and complexity. In particular, we specify the noise as $0.1,$ which is the standard deviation of the Gaussian noise applied to the output $y$, and all other arguments are set as default. Our experiment  focuses on the comparison of three different methods including the bisect method, the single SDP and SOCP methods. Similar as in \citet{bishop2020optimal}, the fake input label $z_i$ is set as
\[
z_i = \max\{y_i,y_{0.25}\},
\]
where $y_{0.25}$ represents the lower quartile (25th percentile) of output $y$. More specifically, if the true  label is greater than or equal to the  threshold $y_{0.25}$, then the label would not be modified. Otherwise, the label would be set as  $y_{0.25}$. In all tests, the parameter $\gamma$ is set as 0.01. More results with $\gamma = 0.1$ can be found in the Appendix.

\begin{table}[!htbp]
        \tiny
        \centering
        \caption{Time (seconds) comparison on synthetic data: $m = 2n$}
        \begin{tabular}{rrrrrrrr}
                \toprule
                  $m$   & $n$ & bisect & sSDP & SOCP  & ratio1 & ratio2& eig \\
                \midrule
                200   & 100  & 4.356   & 0.111   & 0.043 & 101   & 3   &  0.001   \\
                1000  & 500  & 167.732 & 3.997   & 0.099 & 1702  & 41  &  0.020    \\
                2000  & 1000 & 988.675 & 45.984  & 0.178 & 5559  & 259 &  0.085  \\
                4000  & 2000 & 7877.041& 438.487 & 0.536 & 14694 & 818 &  0.441   \\
                8000  & 4000 & -       & 3127.316& 1.478 & -     & 2116&  3.349       \\
                12000 & 6000 & -       & -       & 3.079 & -     & -   &  11.245     \\
                \bottomrule
        \end{tabular}%
        \label{tab:timecompare2}
\end{table}%
\begin{table}[htbp]
        \tiny
        \centering
        \caption{Time (seconds)  comparison on synthetic data:  $m = n$}
        \begin{tabular}{rrrrrrrr}
                \toprule
                $m$   & $n$ & bisect & sSDP & SOCP  & ratio1 & ratio2& eig \\
                \midrule
                100   & 100  & 4.342   & 0.107   & 0.040 & 108   & 3    & 0.001\\
                500   & 500  & 158.304 & 4.142   & 0.072 & 2197  & 57   & 0.018\\
                1000  & 1000 & 990.151& 21.781  & 0.225 & 4408  & 97   & 0.085 \\
                2000  & 2000 & 7667.927& 201.411 & 0.586 & 13094 & 344  & 0.442\\
                4000  & 4000 & -       & 2142.952& 2.485 & -     & 862  & 3.264\\
                6000  & 6000 & -       & -       & 2.876 & -     & -    & 11.117\\
                \bottomrule
        \end{tabular}%
        \label{tab:timecompare1}
\end{table}%

\begin{table}[htbp]
        \tiny
        \centering
        \caption{ Time (seconds)  comparison on synthetic data: $m = 0.5n$}
        \begin{tabular}{rrrrrrrr}
                \toprule
                  $m$   & $n$ & bisect & sSDP & SOCP  & ratio1 & ratio2& eig\\
                \midrule
                50    & 100  & 4.146   & 0.105    & 0.047 & 87   & 2    & 0.001\\
                250   & 500  & 156.018 & 4.471    & 0.078 & 2004  & 57  & 0.021\\
                500   & 1000 & 956.343& 69.267   & 0.189 & 5047 & 366  & 0.080\\
                1000  & 2000 & 7495.735& 177.999  & 0.371 & 20217 & 480  & 0.405\\
                2000  & 4000 & -       & 1485.843 & 1.229 & -     & 1209 & 3.144\\
                3000  & 6000 & -       & 8769.430 & 2.616 & -     & 3352 & 10.436\\
                \bottomrule
        \end{tabular}%
        \label{tab:timecompare05}
\end{table}%
Tables \ref{tab:timecompare2}, \ref{tab:timecompare1} and  \ref{tab:timecompare05} summarise the comparison of wall-clock time on different scales with $m = pn,\ p \in \{0.5,1,2\}$. In these tables, ``bisect" represents the bisection method in \citet{bishop2020optimal},
``sSDP" represents our single SDP method, ``SOCP" represents our SOCP method,
``ratio1" represents the ratio of  times of the bisection method and our SOCP method, and ``ratio2" represents the ratio of  times of our single SDP method and our SOCP method. The last column ``eig" recorded the spectral decomposition time of matrix $\bar A_{11}$ in \eqref{eq:formbarA}. In the test, the algorithm would not be run in larger dimension case (denoted by ``-"), if its wall-clock time at current dimension exceeds 1800 seconds.

From the three tables, we can find that our single SDP method is consistently faster than the bisection method. The ratios in the table also demonstrate the high efficiency of our SOCP method, which can be up to 20,000+ times faster than the bisection method for case $(m,n)=(1000,2000)$. Our SOCP method is also significantly faster than our single SDP method. For example, our SOCP method took about 3 seconds for all cases with $n=6000$, while our single SDP method  took at least 8,000 seconds for the case $(m,n)=(3000,6000)$.
We also remark that the performance gap grows considerably with the problem size since both the ratios increase as the dimension increases.
Finally, we mention that, compared to the time of our single SDP method, the time of spectral decomposition in formulating our SOCP is rather small, which is about 11 seconds for $n=6000$. 

\section{Conclusion}
In this paper, we study the computation for Stackelberg equilibrium of SPG-LSs. Hidden convexity in the fractional programming formulation \eqref{pb:ori} of the SPG-LS is deeply explored. Then, we are able to reformulate the SPG-LS as a single SDP, based on the S-lemma with equality.  By using simultaneous diagonalizability
of its submatrices in the constraint, we further reformulate  our SDP into an SOCP.  We also demonstrate the optimal solution
to the SPG-LS can be recovered easily from solving our obtained SDP or SOCP.
Numerical comparisons between our single SDP and SOCP approaches with the  state of the art demonstrate the high efficiency   as well as learning accuracy of our methods for handling large-scale SPG-LSs.
We believe that our work opens up a new way for the applicability of SPG-LSs in large-scale real scenarios.


\newpage
\onecolumn
\begin{center}
{\Large \bf Supplementary Material}
\end{center}
\par\noindent\rule{\textwidth}{1pt}
\setcounter{section}{0}

\renewcommand\thesection{\Alph{section}}
\renewcommand\thesubsection{\arabic{subsection}}

\section{Recovering Solutions to the SPG-LS}
\subsection{Existence of an Optimal Solution to \eqref{pb:SDP}}
Note that $(\mu,\lambda)=(0,0)$ is a feasible solution to \eqref{pb:SDP} as $$A = \begin{smallpmatrix}X^T\\\z-\y\\-\y\end{smallpmatrix} \begin{smallpmatrix}X & \z-\y&-\y\end{smallpmatrix}\succeq0.$$
Hence, the optimal value of \eqref{pb:SDP} is bounded from above.
Next, consider the following dual problem of \eqref{pb:SDP}
\begin{equation}
\label{pb:dualSDP}
\begin{array}{ll}
\min\limits_W& \langle A,W\rangle\\
\rm s.t.&\langle B,W\rangle=1,\\
&\langle C,W\rangle=0,\\
&W\succeq0.
\end{array}
\end{equation}
Note that $\tilde W=\begin{smallpmatrix}
\frac{\gamma}{8n}I_n&&\\
&\frac{3}{8}&\frac{1}{8}\\
&\frac{1}{8}&\frac{3}{8}
\end{smallpmatrix}$, which satisfies
\[ \tilde W\succ 0,~\langle B,\tilde W\rangle=1\text{ and }\langle C,\tilde W\rangle =0,\]
is a strictly feasible solution to \eqref{pb:dualSDP}. Then, we know from weak duality that the optimal value of \eqref{pb:dualSDP} is also bounded from below, i.e., it is a finite value. Thus there exists an optimal solution to  \eqref{pb:dualSDP} (see, e.g., Theorem 1.4.2 in \citet{ben2012lectures} or Corollary 5.3.10 in \citet{Borwein2006}).

\subsection{Recovering an Optimal Solution to the SPG-LS}
\subsubsection{Recovering an Optimal Solution from the Dual Solution of \eqref{pb:SDP}}
Now let $(\mu^*,\lambda^*)$ be an optimal solution to \eqref{pb:SDP}.
From Lemma \ref{lem:fq}, we know that an optimal solution to \eqref{pb:ori} can be recovered from an optimal solution of the following QCQP
\begin{equation}
\label{pb:gtrs}
\min~ g(\w,\alpha) {\quad \rm s.t.~} \quad \frac{1}{\gamma}\w^T\w=\alpha,
\end{equation}
where  $g(\w,\alpha)=f(\w,\alpha)-\mu^*(1+\alpha)^2$.
By relaxing $\begin{pmatrix}\w\\\alpha\\1 \end{pmatrix}\begin{pmatrix}\w^{T}&\alpha&1 \end{pmatrix}\succeq0$  to $W\succeq0$ and $W_{n+2,n+2}=1$, we have the following standard SDP relaxation of the problem \eqref{pb:gtrs},
\begin{equation}
\label{pb:sdr}
\begin{array}{ll}
\min\limits_W& \langle A-\mu^*B,W\rangle\\
\rm s.t.&\langle C,W\rangle =0,\\
&W_{n+2,n+2}=1,\\
&W\succeq0.
\end{array}
\end{equation}
Note that the dual problem of \eqref{pb:sdr} is
\begin{align}
\label{pb:dualsdr}
\begin{array}{lll}
&\sup\limits_{\lambda,\tau}& \tau\\
&\rm s.t.&A-\mu^{*} B+\lambda C-\tau E\succeq0,
\end{array}
\end{align}
where $E=\Diag({\bm 0}_{n+1},1)$. From Lemma \ref{lem:fq}, we know that the objective value of \eqref{pb:gtrs} is exactly 0 and thus the optimal value of its SDP relaxation \eqref{pb:sdr} is also non-positive. Then, weak duality implies that the optimal value of \eqref{pb:dualsdr} is non-positive. Since $(\lambda,\tau)=(\lambda^*,0)$ is a feasible solution to  \eqref{pb:dualsdr}, it holds that 0 is the optimal value of  \eqref{pb:dualsdr}.
Moreover, we know that the optimal value of \eqref{pb:sdr} is 0.


For problem \eqref{pb:sdr}, one may check that $\tilde W=\begin{smallpmatrix}
\frac{\gamma}{2n}I_n&&\\
&1&\frac{1}{2}\\
&\frac{1}{2}&1
\end{smallpmatrix}$ is a strictly feasible solution.
We assume \eqref{pb:dualsdr} is also strictly feasible.
From Theorem 1.4.2 or Section 3.1.1.2 in \citet{ben2012lectures}, we know that both \eqref{pb:sdr} and \eqref{pb:dualsdr} have optimal solutions and any primal dual solution pairs satisfy the KKT optimality condition.
Now suppose $W^*$ is an optimal solution to \eqref{pb:sdr}. Since $(\lambda,\tau)=(\lambda^*,0)$ is an optimal solution to \eqref{pb:dualsdr}, we have
\begin{equation}
\label{eq:KKTsdr}
\begin{array}{l}
\langle C,W^*\rangle=0,\\
W^*_{n+2,n+2}=1,\\
W^*\succeq 0, \\
\langle A-\mu^{*} B+\lambda^* C,W^* \rangle  =0.
\end{array}
\end{equation}

Next we show such a $W^*$ can be recovered from a dual solution of SDP  \eqref{pb:SDP}.
Recall
\[
\tilde B=\begin{pmatrix}
{\bm 0}_{n+1}&\\
&4
\end{pmatrix}, ~~
\tilde C=\begin{pmatrix}
\frac{1}{\gamma}I_{n+1}&\\
&-1
\end{pmatrix}\]
By setting $\mu=-1$ and $\lambda=1$, we have
$$\tilde A-\mu \tilde B+\lambda \tilde C=\tilde A+ \begin{pmatrix}
\frac{1}{\gamma}I_{n+1} &0\\
0&3\\
\end{pmatrix} \succ0,$$
which is equivalent to $A-\mu B+\lambda C\succ0$. Thus we see that there exists a strictly feasible solution for \eqref{pb:SDP}.
Since both \eqref{pb:SDP} and its dual \eqref{pb:dualSDP} are strictly feasible and $(\mu^*, \lambda^*)$ solves \eqref{pb:SDP}, we know by \citet{ben2012lectures} that there exists an optimal solution $\hat W$ to  \eqref{pb:dualSDP}, which satisfies
\[
\langle B,\hat W\rangle=1,~\langle C,\hat W\rangle=0,~\hat W\succeq0~\text{and}~ \langle A-\mu^{*} B+\lambda^{*} C,\hat W\rangle=0.
\]
We assume $\hat W_{n+2,n+2}\neq0$\footnote{One can expect that this is always the case in real applications. Indeed, one may verify from the complementary slackness condition that $\hat W=\Diag({\bm 0}_n,1,0)$ if $\hat W_{n+2,n+2}=0$, and consequently we must have $X^T(\y - \z) = 0$. However, this condition holds with probability $0$ under the assumption that $\y$ and $\z$ generated from some reasonable distribution.}.
It is easy to verify that  $\bar W=\hat W/\hat W_{n+2,n+2}$  is an optimal solution to \eqref{pb:sdr} as $\bar W$ is feasible and, together with the dual solution $(\lambda,\tau)=(\lambda^*,0)$, satisfies \eqref{eq:KKTsdr}.

Now we will introduce the well known rank-1 decomposition for a positive semidefinite matrix in \citet{sturm2003cones} to obtain an optimal solution for \eqref{pb:gtrs}.
\begin{lemma}[Proposition 4 in \citet{sturm2003cones}]
\label{lem:rank}
Let $X$ be a positive semidefinite matrix of rank $r$ in $\R^{m\times m}$. Let $G$ be a given matrix. Then $\langle G,X\rangle=0$ if and only if there exist $\p_i\in\R^m$, $i=1,\ldots,r$, such that
\begin{equation}
\label{eq:rank}
X=\sum_{i=1}^r\p_i\p_i^T\text{ and }\p_i^TG\p_i=0\quad \text{for all }~i=1,2,\ldots,r.
\end{equation}
\end{lemma}
\begin{algorithm}[!hptb]
\caption{Decomposition of $X$ satisfying  \eqref{eq:rank}}
\label{alg:rank}
\begin{algorithmic}[1]
\STATE {\bfseries Input:} positive semidefinite matrix  $X$ with rank $r$ in $\R^{m\times m}$,  a given matrix $G$ such that $\langle G,X\rangle=0$
\STATE set $\bar A,\bar B,\bar C$ as in \eqref{eq:formbarA}, \eqref{eq:formbarB}, \eqref{eq:formbarC}
\STATE do spectral decomposition to matrix $\bar A_{11}$ in \eqref{eq:formbarA} with $\bar{A}_{11} = HDH^{T}$
\WHILE{there exist $\p_i , \p_j$ such that $\p_i^TG\p_i>0,\p_j^TG\p_j<0$} \STATE compute a root $t$ of the quadratic equation
\[\p_i^TG\p_i t^2 + 2\p_i^TG\p_jt+\p_j^TG\p_j=0\]
\STATE set $\tilde \p_i=\frac{1}{\sqrt{t^2+1}}(t\p_i+\p_j)$,  $\tilde \p_{j}=\frac{1}{\sqrt{t^2+1}}(t\p_j+\p_i)$ \STATE set $\p_i=\tilde \p_i$, $\p_j=\tilde \p_j$
\ENDWHILE
\STATE return $X=\sum_{i=1}^r\p_i\p_i^T$
\end{algorithmic}
\end{algorithm}
We next adapt Algorithm \ref{alg:rank} for computing a decomposition for $X$ in Lemma \ref{lem:rank}, which is a variant of Algorithm 2 in \citet{hazan2016linear}.
Similar to Lemma 1 in  \citet{hazan2016linear},  the while loop ends in at most $r$ steps.

Using Algorithm \ref{lem:rank} and setting $G = C$, we obtain a decomposition  $W^*=\sum_{i=1}^r\p_i\p_i^T$ satisfying
$$ \p_i^T C \p_i =0 \quad \mbox{ for all }~ i= 1,\ldots, r.$$
Since $\langle A-\mu^*B+\lambda^* C,W^*\rangle=0$, we  have $\sum_{i=1}^{r}\p_i^T( A-\mu^*B+\lambda^* C)\p_i^T=0$. 
This, together with the positive semidefiniteness of $ A-\mu^*B+\lambda^* C$, implies
$$\p_i^T( A-\mu^*B+\lambda^* C)\p_i =0\quad\text{for all  }~i=1,\ldots,r.$$
Then, we know that $\p_i^T( A-\mu^*B)\p_i =0$ for all $i=1,\ldots, r$.
Since $W_{n+2,n+2}^*=1$, there exists some $j\in [r]$ such that  $(\p_j)_{n+2}\neq0$.
Let $\hat \w= \frac{(\p_j)_{1:n+1}}{(\p_j)_{n+2}}$, then we obtain a rank-1 solution
$$ \hat W = \begin{pmatrix}\hat \w\\1 \end{pmatrix}\begin{pmatrix}\hat \w^{T}&1 \end{pmatrix}$$
to problem \eqref{pb:sdr}.
Let $\w^* = \hat \w_{1:n} \in \R^n$ and $\alpha^* = \hat \w_{n+1} \in \R$. It is not difficult to verify that $\alpha^* = \w^{*T} \w^*/\gamma$ and $g(\w, \alpha) =0$, i.e., $(\w, \alpha) $ solves problem  \eqref{pb:gtrs}. Then,  we know from  Lemma \ref{lem:fq} that $\w^*$ is an optimal solution to \eqref{pb:ori}.

We remark here that  in our recovering phase, no additional SDP needs to be solved. Indeed, in most interior point methods based solvers for SDPs (e.g., MOSEK \cite{aps2021mosek}, SDPT3 \cite{toh1999sdpt3}), both primal and dual solutions of SDPs are computed simultaneously.
Hence, to obtain a solution $\hat W$ to \eqref{pb:dualSDP}, we only needs to solve either \eqref{pb:SDP} or \eqref{pb:dualSDP}.

\subsubsection{Recovering an Optimal Solution from the Solution of SOCP \eqref{pb:SOCP}}

Next we show the correctness of step 8 in Algorithm \ref{alg:socp}.
The analysis in the previous subsection reveals that there exists a rank-1 solution $\hat W$ to \eqref{pb:sdr}, and $(\lambda,\tau)=(\lambda^*,0)$ is a solution to the dual problem
\eqref{pb:dualsdr}. Given the constraint $W_{n+2,n+2}=1$, we may assume $\hat W=\begin{smallpmatrix}
        \hat \w \\
        1
        \end{smallpmatrix}\begin{smallpmatrix}
        \hat \w^T&
        1
        \end{smallpmatrix}$.
Then from \eqref{eq:KKTsdr}, we arrive at the following sufficient and necessary condition in terms of $\hat \w$,
\begin{equation*}
\begin{array}{l}
\langle C,\begin{smallpmatrix}
        \hat \w \\
        1
        \end{smallpmatrix}\begin{smallpmatrix}
        \hat \w^T&
        1
        \end{smallpmatrix}\rangle=0,\\
\langle A-\mu^{*} B+\lambda^{*} C,\begin{smallpmatrix}
        \hat \w \\
        1
        \end{smallpmatrix}\begin{smallpmatrix}
        \hat \w^T&
        1
        \end{smallpmatrix} \rangle  =0,
\end{array}
\end{equation*}
Let $\hat \w= \begin{smallpmatrix}
        \w \\
        \alpha
        \end{smallpmatrix}. $
Then, the first equation implies $\frac{1}{\gamma}\w^T\w=\alpha$,
and the second equation, together with the positive semidefiniteness of  $A-\mu^{*} B+\lambda^{*} C$, implies  \[(A-\mu^{*} B+\lambda^{*} C)\begin{smallpmatrix}
        \w \\
        \alpha\\
        1
        \end{smallpmatrix}=0.\]
Since the existence of such a $(\w,\alpha)$ is guaranteed from the results in the previous subsection, when the linear system has a unique solution\footnote{A sufficient condition to guarantee this is that the matrix $(A-\mu^{*} B+\lambda^{*} C)$ is of rank $n+1$, which is observed in all our numerical tests.}, it suffices to solve the linear system solely.



\newpage
\section{Additional Experiments}
In this section, we compare our single SDP and SOCP methods with existing algorithms for additional real and synthetic datasets. { Besides, to validate the robustness of our model, we add experiments with random noises in the target label as well.}

\subsection{Real World Dataset}
We illustrate the accuracy and efficiency of our methods on two additional real world datasets, the insurance dataset\footnote{https://www.kaggle.com/mirichoi0218/insurance/metadata}
and the residential building dataset\footnote{https://archive.ics.uci.edu/ml/datasets/Residential+Building+Data+Set}.
\subsubsection{The Insurance Dataset}
The insurance dataset consists of 1338 instances with 7 features, each regarding to certain information of an individual such as age, region and smoking status. For the test purpose, we transform the categorical features  into a one-hot vector.
Similar as in \citet{bishop2020optimal}, we consider the scenario with an insurer and multiple individuals.
The insurer collects data-form information from the individuals to predict future insurance quote while the latter provide fake data in order to make the quote lower.
Corresponding to this situation, we define the individual's desired outcome as
$$z_i = \max\{y_i +\delta,0\}, $$
where $\delta=-100$ in the modest case and $\delta=-300$ in the severe case.

All the hyperparameters are the same as those used in Section \ref{sec:5.1}.
The MSE and the computational time comparisons are illustrated in Figure \ref{insurancemsetime}.   As one can observe, the bisection method, our single SDP and our SOCP methods achieved the best performance in terms of MSE.
On average, our SOCP method, for both the modest and severe cases, is about 40 times faster than the bisection method, and is slightly better than our single SDP method.
These results again verify the accuracy and efficacy of our methods.


\begin{figure}[htbp]
\centering
\subfigure{
\includegraphics[scale=0.27]{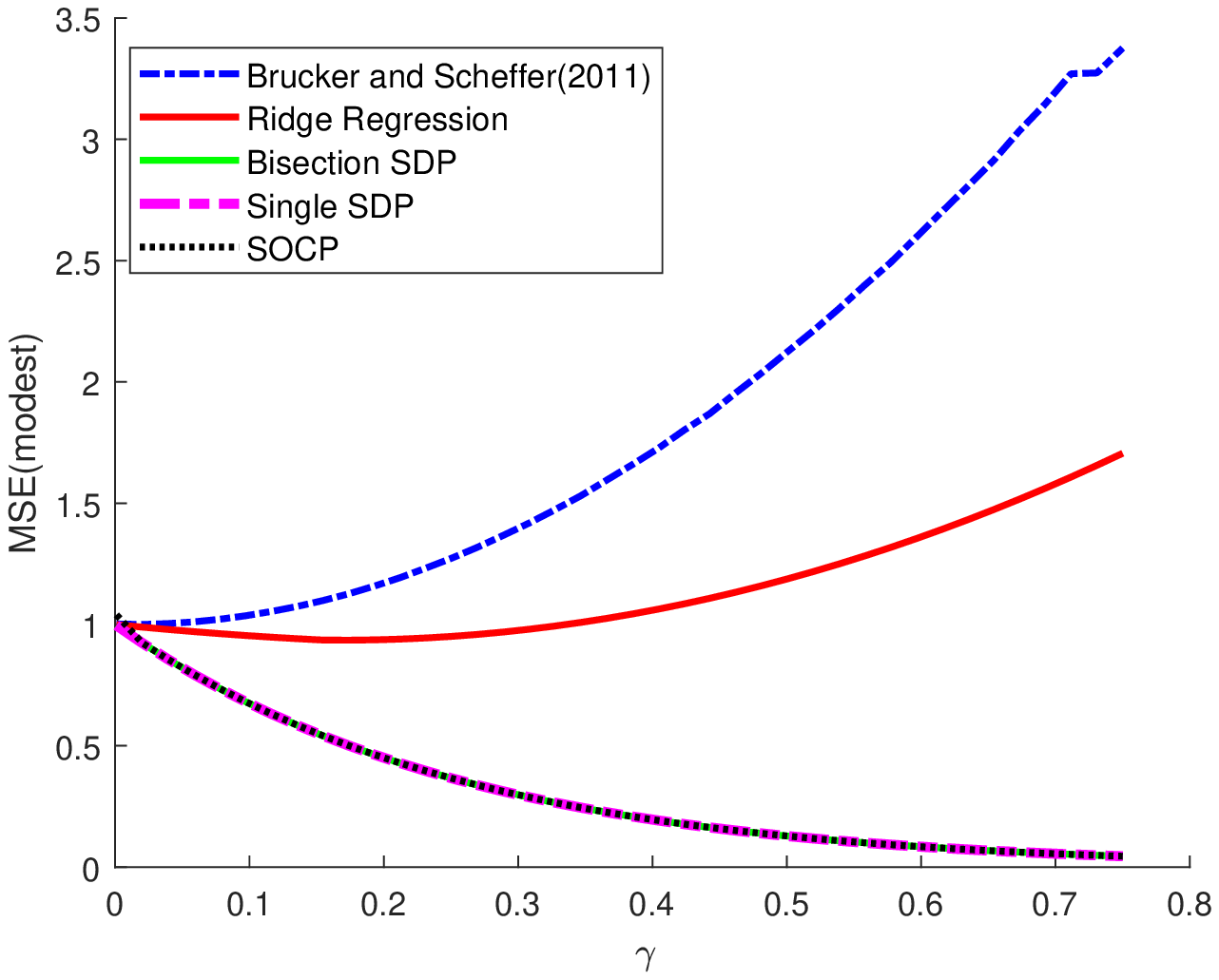} \label{fig:insurmodestmse}
}
\subfigure{
\includegraphics[scale=0.27]{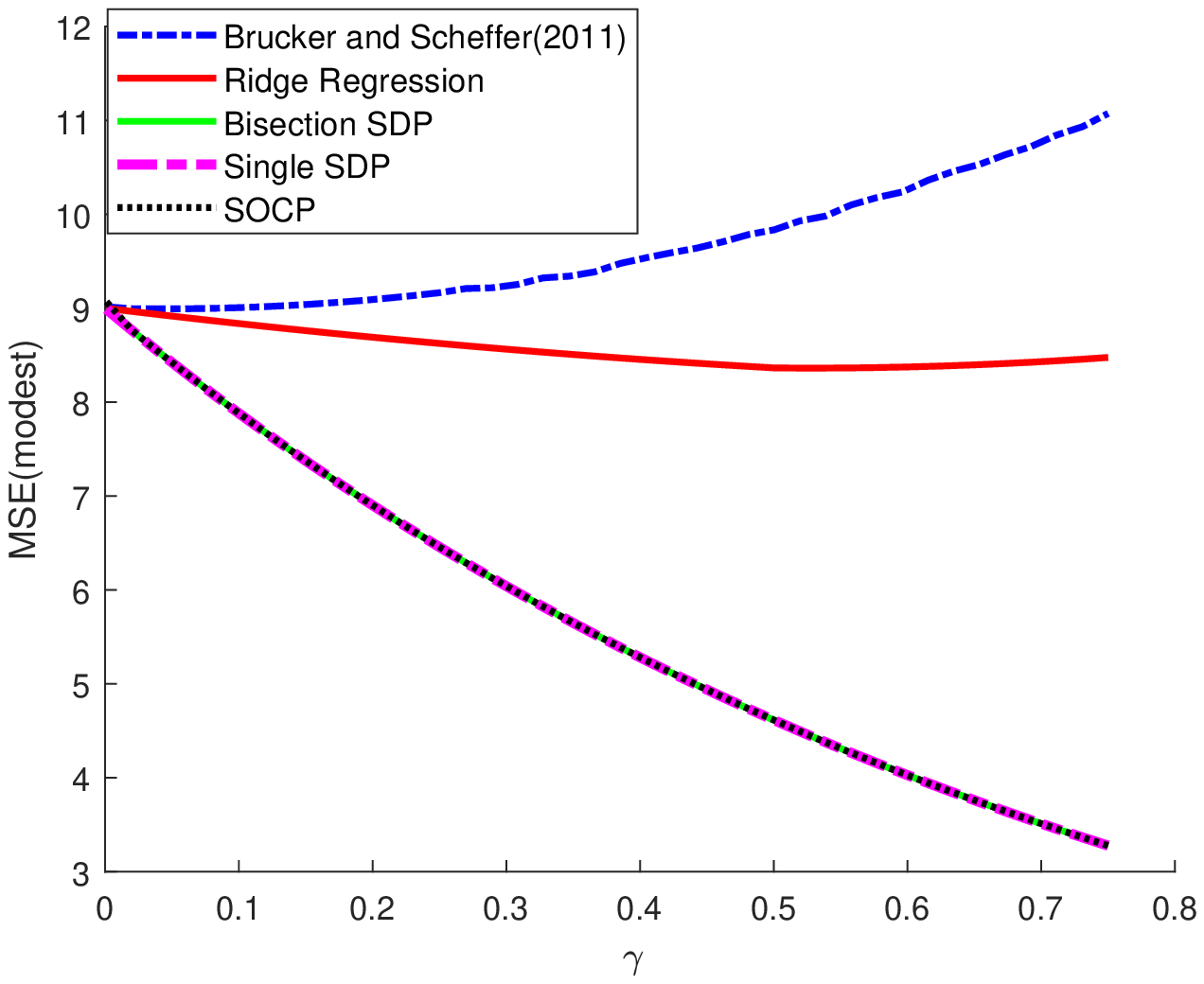} \label{fig:insurseveremse}
}
\subfigure{
\includegraphics[scale=0.27]{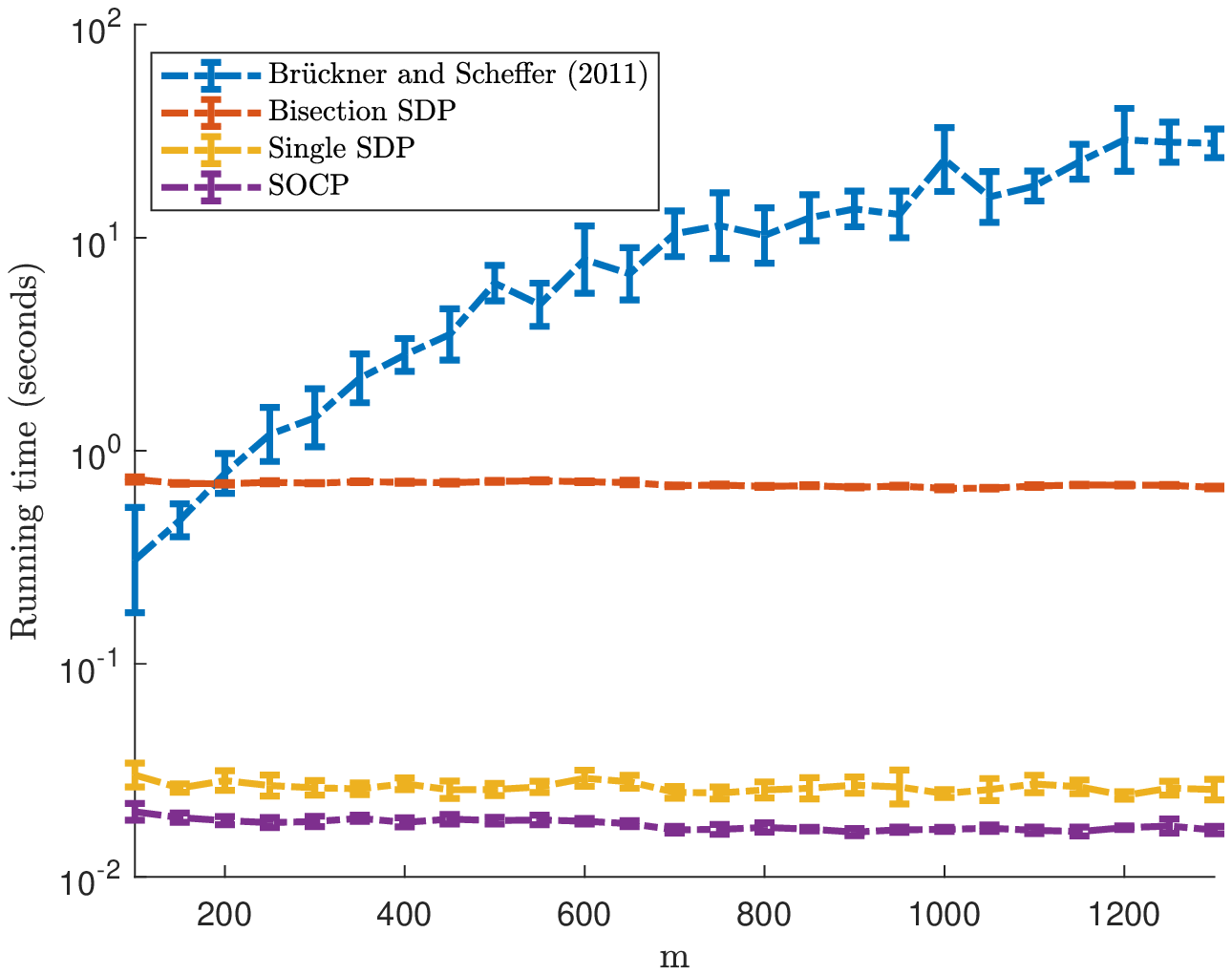} \label{fig:insurmodestime}
}
\subfigure{
\includegraphics[scale=0.27]{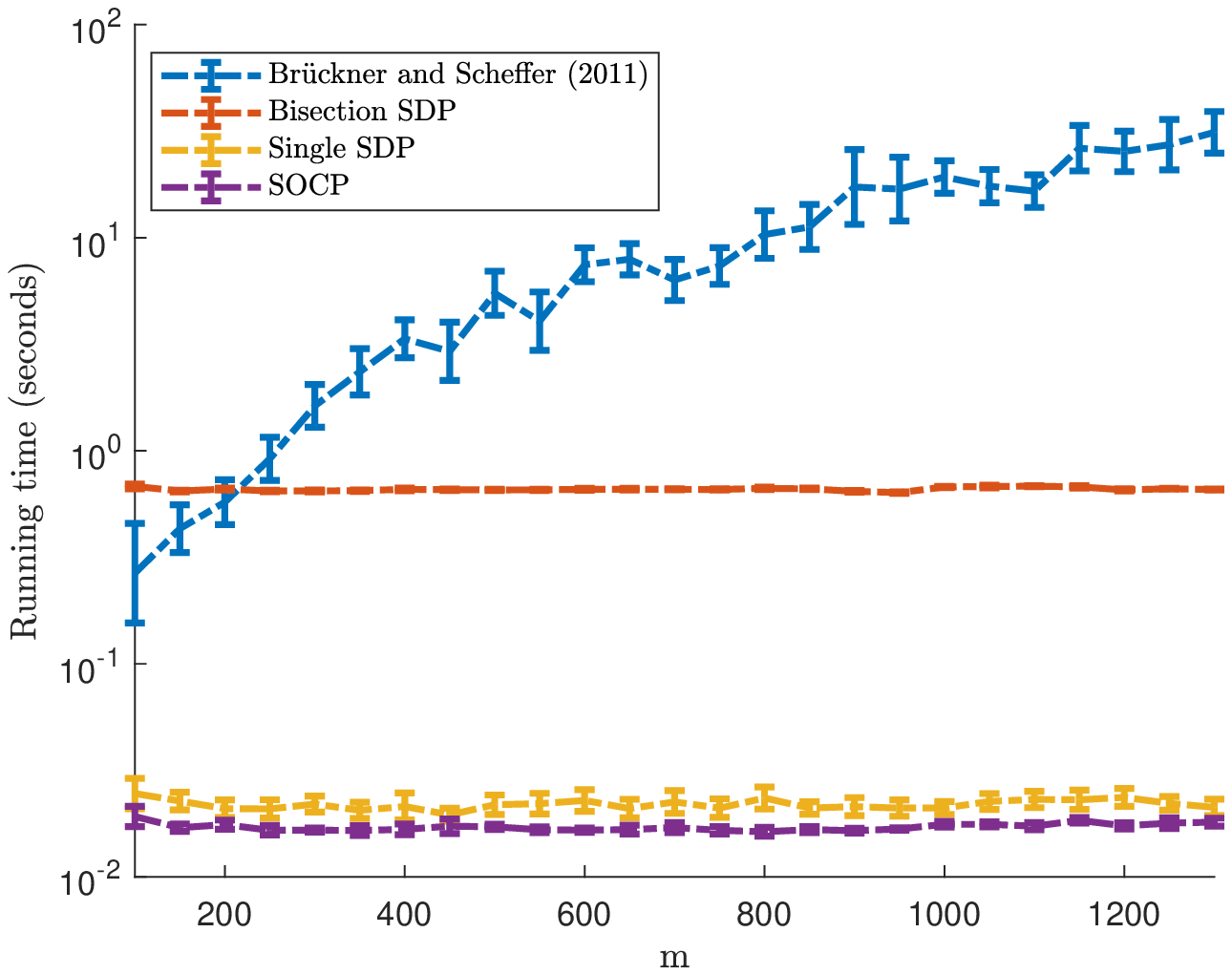} \label{fig:insurseveretime}}
\caption{Performance comparison between different algorithms on the insurance dataset.  The left two plots correspond to MSE result generated  by $\mathcal A_{\rm modest}$ and $\mathcal A_{\rm severe}$, whilst the right two plots correspond to wall-clock time comparison generated  by $\mathcal A_{\rm modest}$ and $\mathcal A_{\rm severe}$.}
\label{insurancemsetime}
\end{figure}

\subsubsection{The Residential Building Dataset}
The residential building data set consists of 372 instances each with 107 features. The response variable is chosen to be the actual sales prices. We consider a scenario in which sellers want to sell the buildings at a higher price and buyers try to predict fair prices. We define the seller's desired outcome
$$
    z_i = y_i + \delta, 
$$
where  $\delta = 20$ for $\mathcal A_{\rm modest}$ and $\delta = 40$ for  $\mathcal A_{\rm severe}$.
All the hyperparameters are the same as those used in Section \ref{sec:5.1}.

The resulted MSE of the experiments and  the running time comparison are shown in Figure \ref{buildingmsetime}. Similar as in the previous case, our algorithms outperform other approaches in terms of the MSE and running time. On average, our single SDP method is about 25  times (30 times, respectively) faster than the bisection in the modest case (the severe case, respectively), while our SOCP method is about 500  times (550 times, respectively) faster than the bisection in the modest case (the severe case, respectively).  As for the MSE, we observe that for both types of data providers, the predictions made by our algorithms are much more accurate than the ridge regression, and the same accurate as the bisection method.

\begin{figure}[htbp]
\centering
\subfigure{
\includegraphics[scale=0.27]{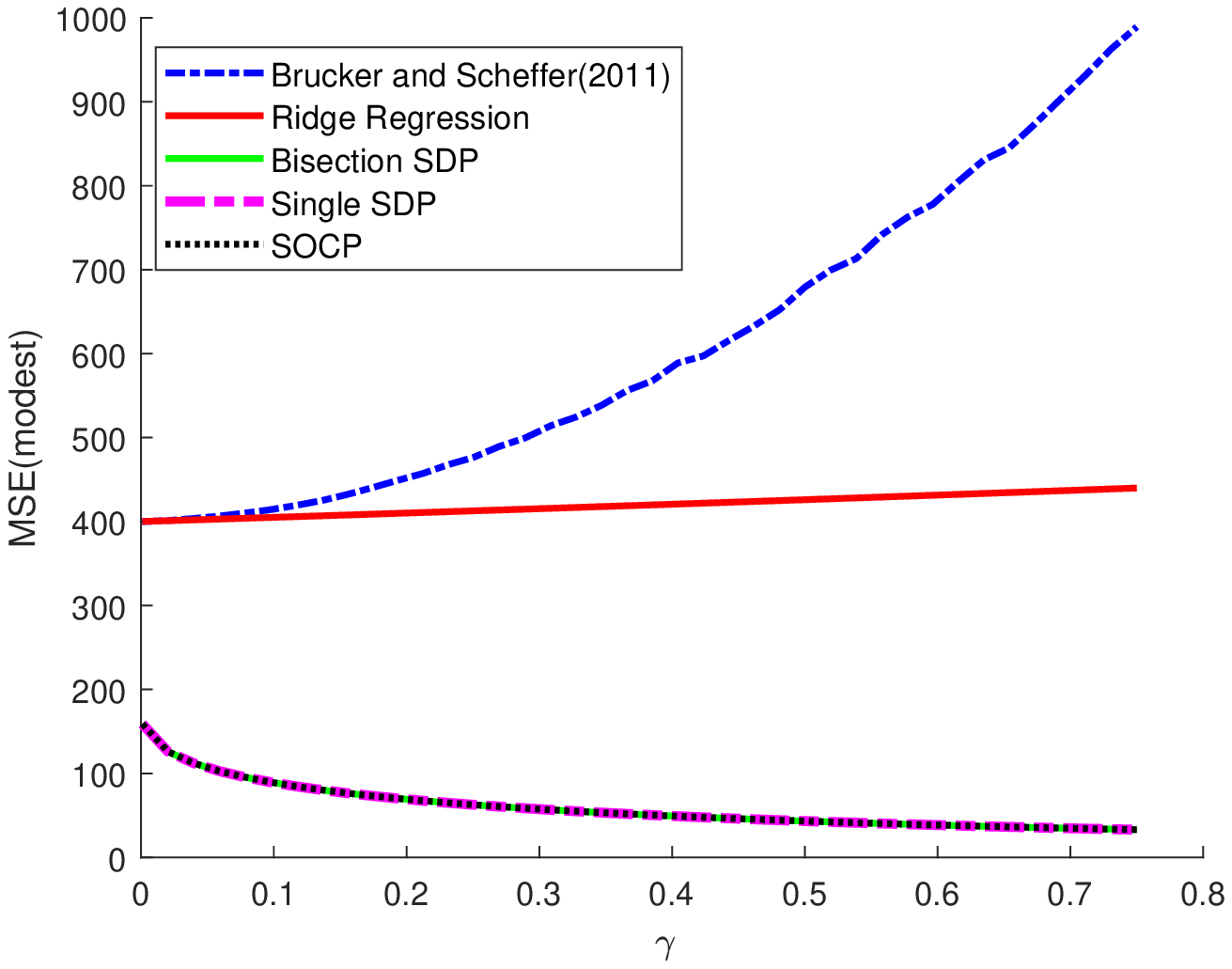} \label{fig:buildmodestmse}
}
\subfigure{
\includegraphics[scale=0.27]{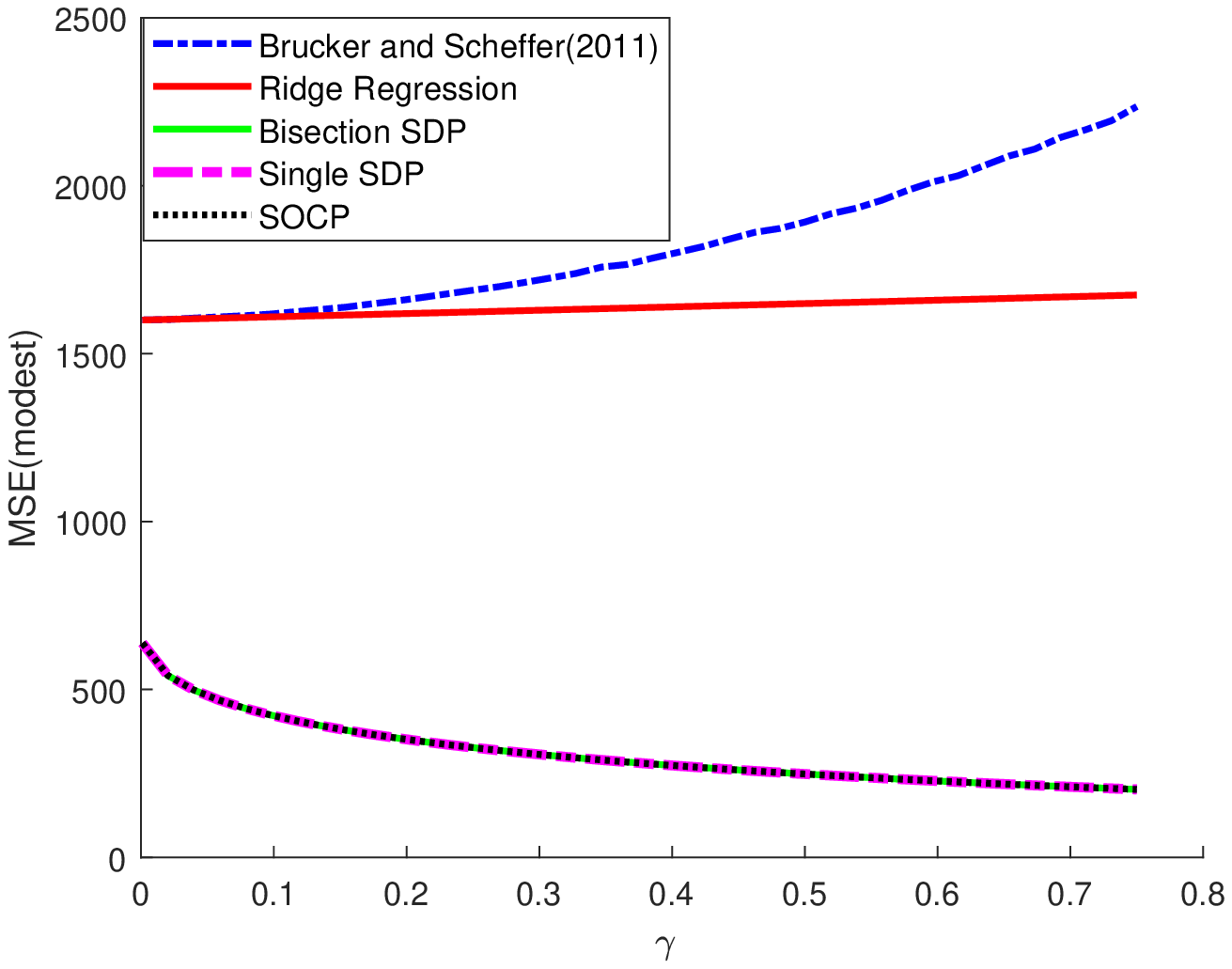} \label{fig:buildseveremse}
}
\subfigure{
\includegraphics[scale=0.27]{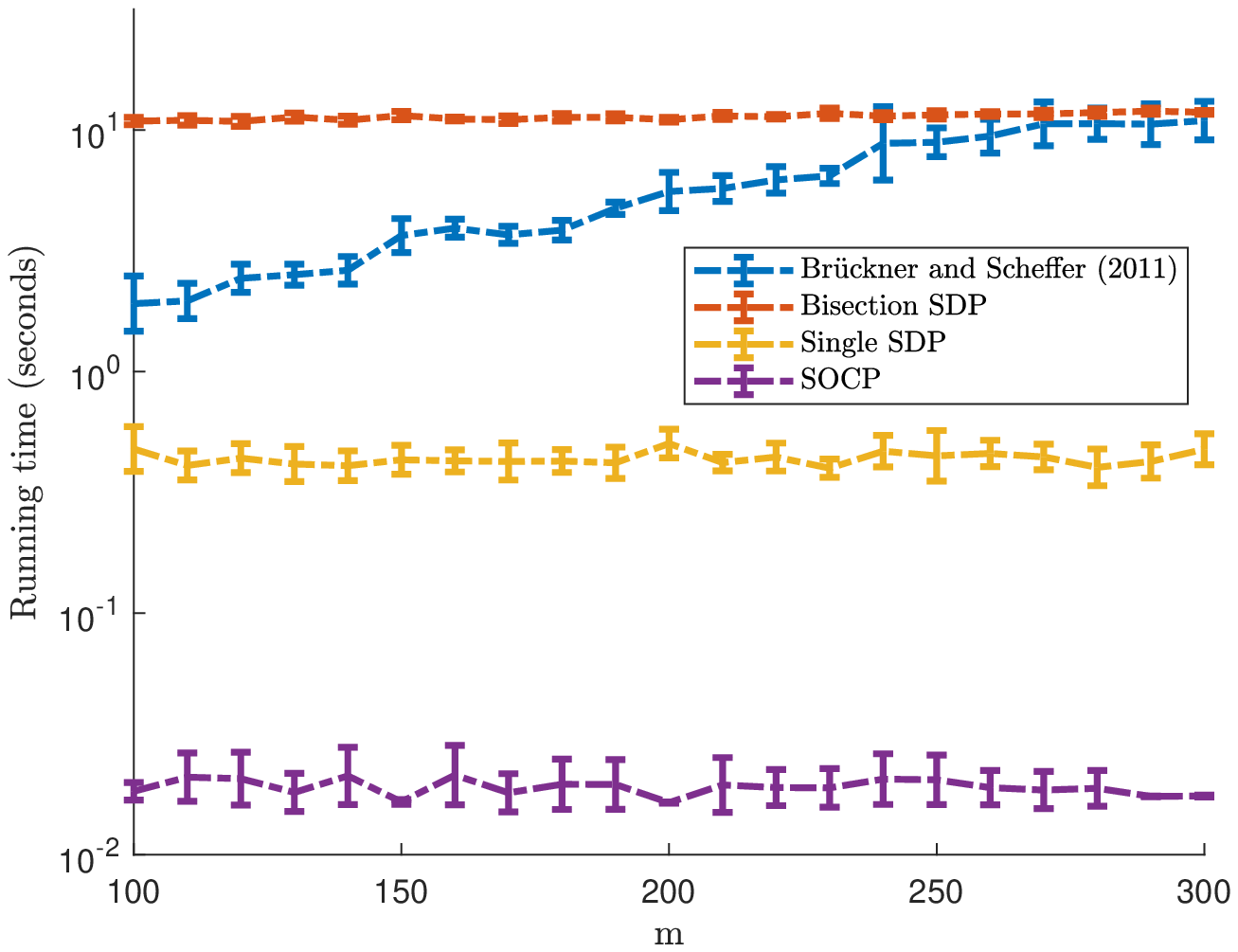} \label{fig:buildmodestime}
}
\subfigure{
\includegraphics[scale=0.27]{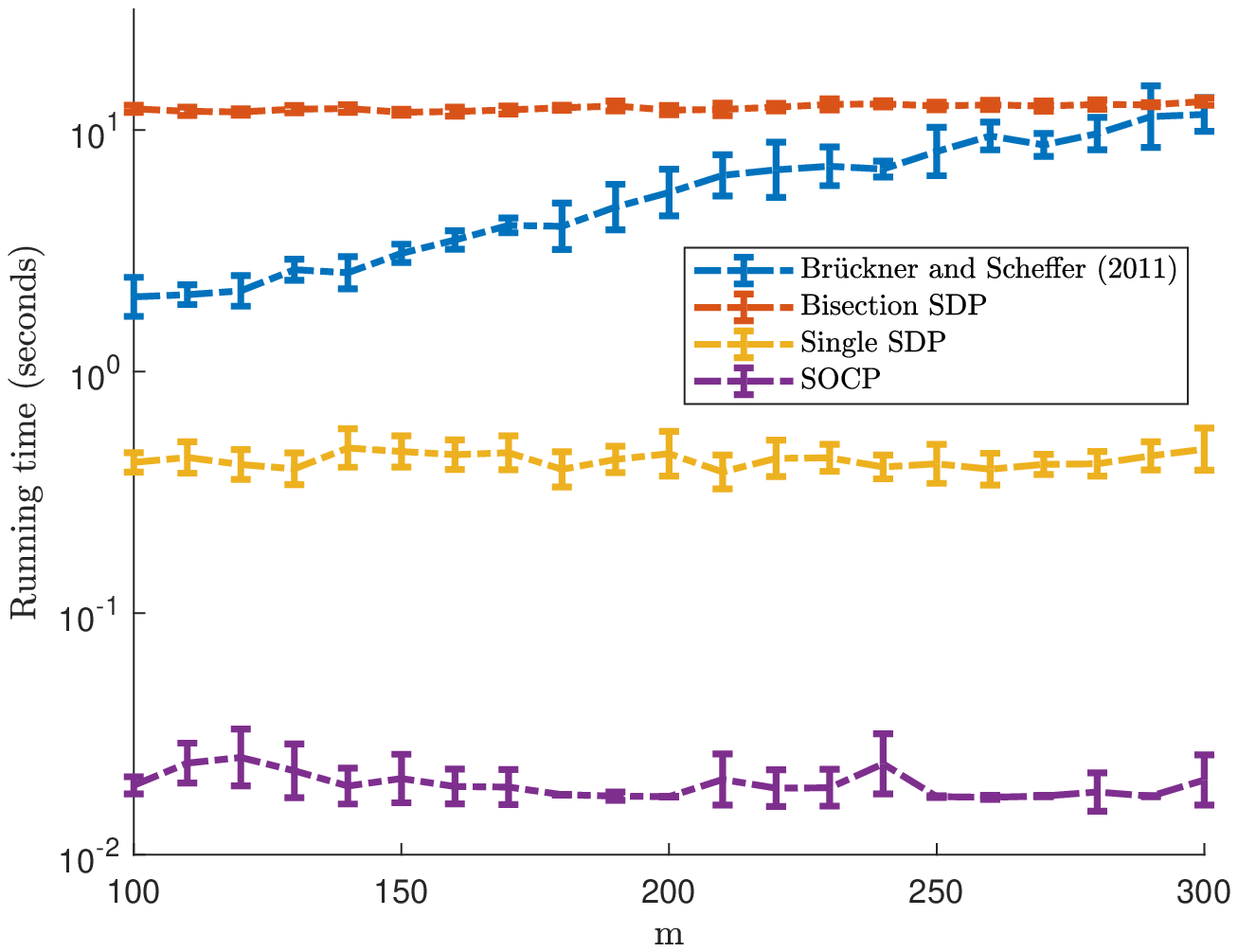} \label{fig:buildseveretime}}
\caption{Performance comparison between different algorithms on the building dataset.  The left two plots correspond to MSE result generated  by $\mathcal A_{\rm modest}$ and $\mathcal A_{\rm severe}$, whilst the right two plots correspond to wall-clock time comparison generated  by $\mathcal A_{\rm modest}$ and $\mathcal A_{\rm severe}$.}
\label{buildingmsetime}
\end{figure}

\subsection{Synthetic Dataset}

%

To further demonstrate the efficiency of our approaches, we perform experiments with $\gamma = 0.1$ on more synthetic data sets. All the other settings are the same as in Section \ref{sec:5.2}.

From Tables \ref{tab:timecompare2_ap}-\ref{tab:timecompare05_ap}, the numerical experiments demonstrated the superiority of our SOCP method over others, as observed in Section \ref{sec:5.2}.
When the dimension of $n$ is small, both single SDP and SOCP methods are efficient. However, when $n$ increases, our SOCP method performs much better than the single SDP method.
Compared with the bisection method in the case $(m,n)=(1000,2000)$, our SOCP can be up to 17,000+ times faster. Moreover,
our SOCP method took less than 3 seconds for all cases with $n = 6000$, while our single SDP took more than 10,000 seconds for the case $(m,n) = (3000,6000)$. We also remark that the performance gap grows considerably with the problem size since both the ratios increase as the dimension increases (except for ratio2 of the instance $(m, n) = (4000, 2000)$). Finally, the time of spectral decomposition in formulating our SOCP is quite small, which is less than 11 seconds for $n = 6000$.
\begin{table}[!htbp]
        \centering
        \caption{Time (seconds) comparison on synthetic data: $m = 2n$}
        \begin{tabular}{rrrrrrrr}
                \toprule
                  $m$   & $n$ & bisect & sSDP & SOCP  & ratio1 & ratio2& eig \\
                \midrule
                200   & 100  & 4.856   & 0.148   & 0.042 & 117   & 4   &  0.001   \\
                1000  & 500  & 175.010 & 3.854   & 0.112 & 1565  & 34  &  0.020    \\
                2000  & 1000 & 1166.041& 62.065  & 0.154 & 7559  & 409 &  0.084  \\
                4000  & 2000 & 9268.016& 183.295 & 0.556 & 16683 & 330 &  0.455   \\
                8000  & 4000 & -       & 2372.122& 1.420 & -     & 1670&  3.330       \\
                12000 & 6000 & -       & -       & 2.783 & -     & -   &  10.943     \\
                \bottomrule
        \end{tabular}%
        \label{tab:timecompare2_ap}
\end{table}%
\begin{table}[htbp]
        \begin{center}
        \caption{Time (seconds)  comparison on synthetic data:  $m = n$}
        \begin{tabular}{rrrrrrrr}
                \toprule
                $m$   & $n$ & bisect & sSDP & SOCP  & ratio1 & ratio2& eig \\
                \midrule
                100   & 100  & 4.885   & 0.127   & 0.024 & 200   & 5    & 0.001\\
                500   & 500  & 173.118 & 4.408   & 0.046 & 3798  & 97   & 0.019\\
                1000  & 1000 & 1130.008 & 47.321  & 0.173 & 6542  & 274   & 0.083 \\
                2000  & 2000 & 8547.944& 334.814 & 0.476 & 17955 & 703  & 0.460\\
                4000  & 4000 & -       & 2547.903& 1.588 & -     & 1604  & 3.319\\
                6000  & 6000 & -       & -       & 2.697 & -     & -    & 10.880\\
                \bottomrule
        \end{tabular}%
        \label{tab:timecompare1_ap}
        \end{center}
\end{table}%

\begin{table}[htbp]
        \begin{center}
        \caption{ Time (seconds)  comparison on synthetic data: $m = 0.5n$}
        \begin{tabular}{rrrrrrrr}
                \toprule
                  $m$   & $n$ & bisect & sSDP & SOCP  & ratio1 & ratio2& eig\\
                \midrule
                50    & 100  & 4.571   & 0.131    & 0.038 & 121   & 3    & 0.001\\
                250   & 500  & 167.787 & 8.960    & 0.119 & 1411  & 75  & 0.020\\
                500   & 1000 & 1039.244& 37.309   & 0.135 & 7702 & 277  & 0.073\\
                1000  & 2000 & 8397.725& 296.672  & 0.523 & 16052 & 567  & 0.378\\
                2000  & 4000 & -       & 1652.523 & 1.518 & -     & 1088 & 3.121\\
                3000  & 6000 & -       & 10026.490& 2.550 & -     & 3932 & 10.340\\
                \bottomrule
        \end{tabular}%
        \label{tab:timecompare05_ap}
        \end{center}
\end{table}%
\subsection{Random Noise}
To see how the function with random noises can affect the performance, we add the following experiments in wine dataset. Specifically, we add Guaussian noises to the target labels, i.e., $t_i=t+w_i$ for $w_i\sim N(0,\sigma^2)$.
We also do experiments with the truncation threshold { randomly changing in some interval}, i.e., $t_i=t+w_i$ for $w_i$ uniformly sampled in some interval centered at the origin. The MSE results of different algorithms in wine dataset generated by $\mathcal A_{\rm modest}$ and $\mathcal A_{\rm severe}$ are showed in Figure \ref{fig:noisewinemodest} and \ref{fig:noisewinesevere} respectively. The figures indicate that the model performance varies when the truncation threshold changes with respect to a Gaussian noise  $N(0,0.5^2)$ (truncated back to the interval $[y_i,10]$) or a uniform noise in $[t-1,t+1]$.
In all cases, the three global methods { (i.e., the Bisection, single SDP and SOCP methods)} achieve the best MSEs.
These results verify the robustness of our single SDP and SOCP model. The comparisons of runtime
are not presented here as they are similar to those presented
in the current paper.

\begin{figure}[htbp]
\centering     
\subfigure{\includegraphics[width=0.32\textwidth]{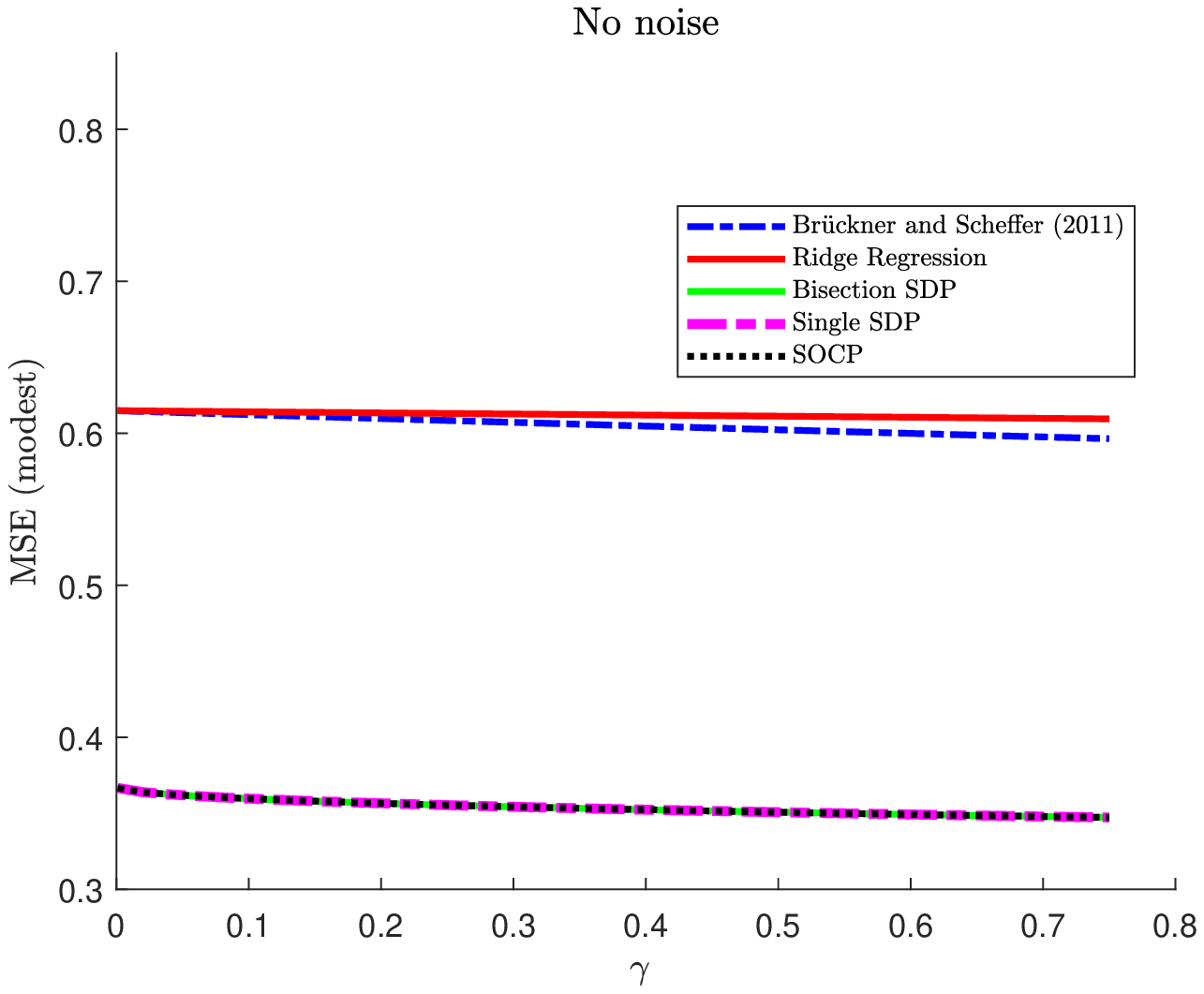}}
\subfigure{\includegraphics[width=0.32\textwidth]{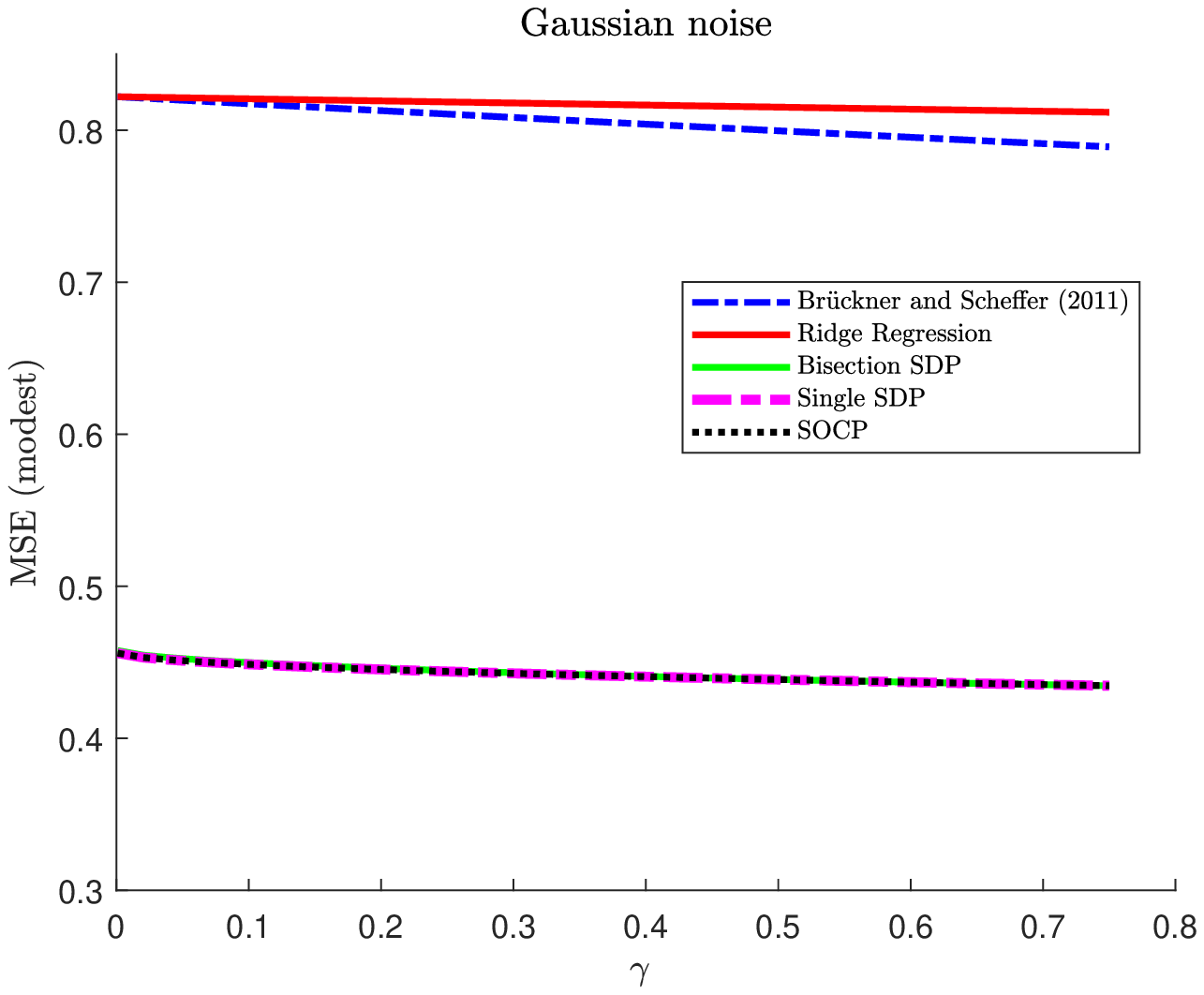}}
\subfigure{\includegraphics[width=0.32\textwidth]{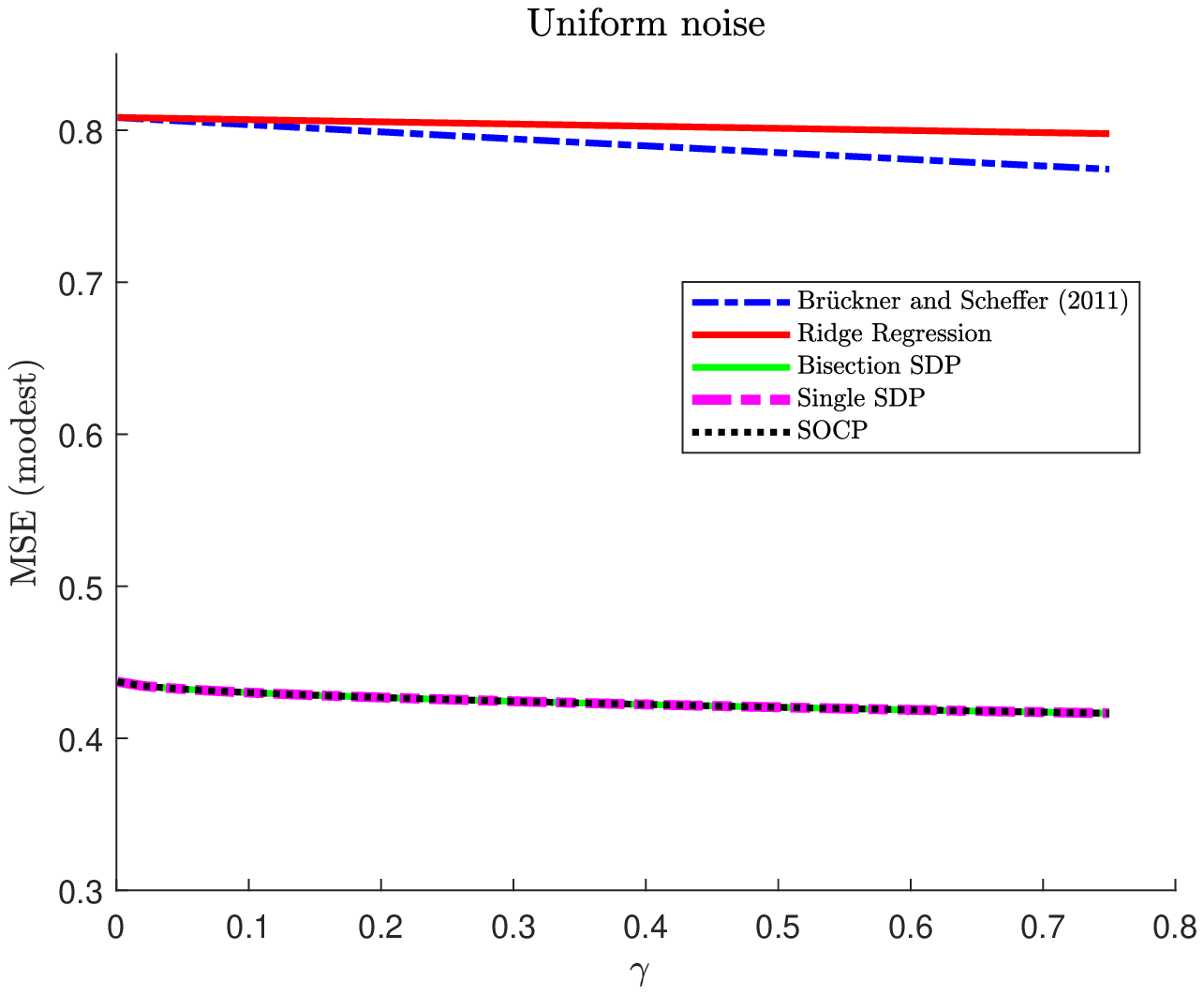}}
\vspace{-5mm}
\caption{MSE comparison for wine dataset generated by $\mathcal A_{\rm modest}$.}
\label{fig:noisewinemodest}
\vspace{-4mm}
\end{figure}

\begin{figure}[htbp]

\centering     
\subfigure{\includegraphics[width=0.32\textwidth]{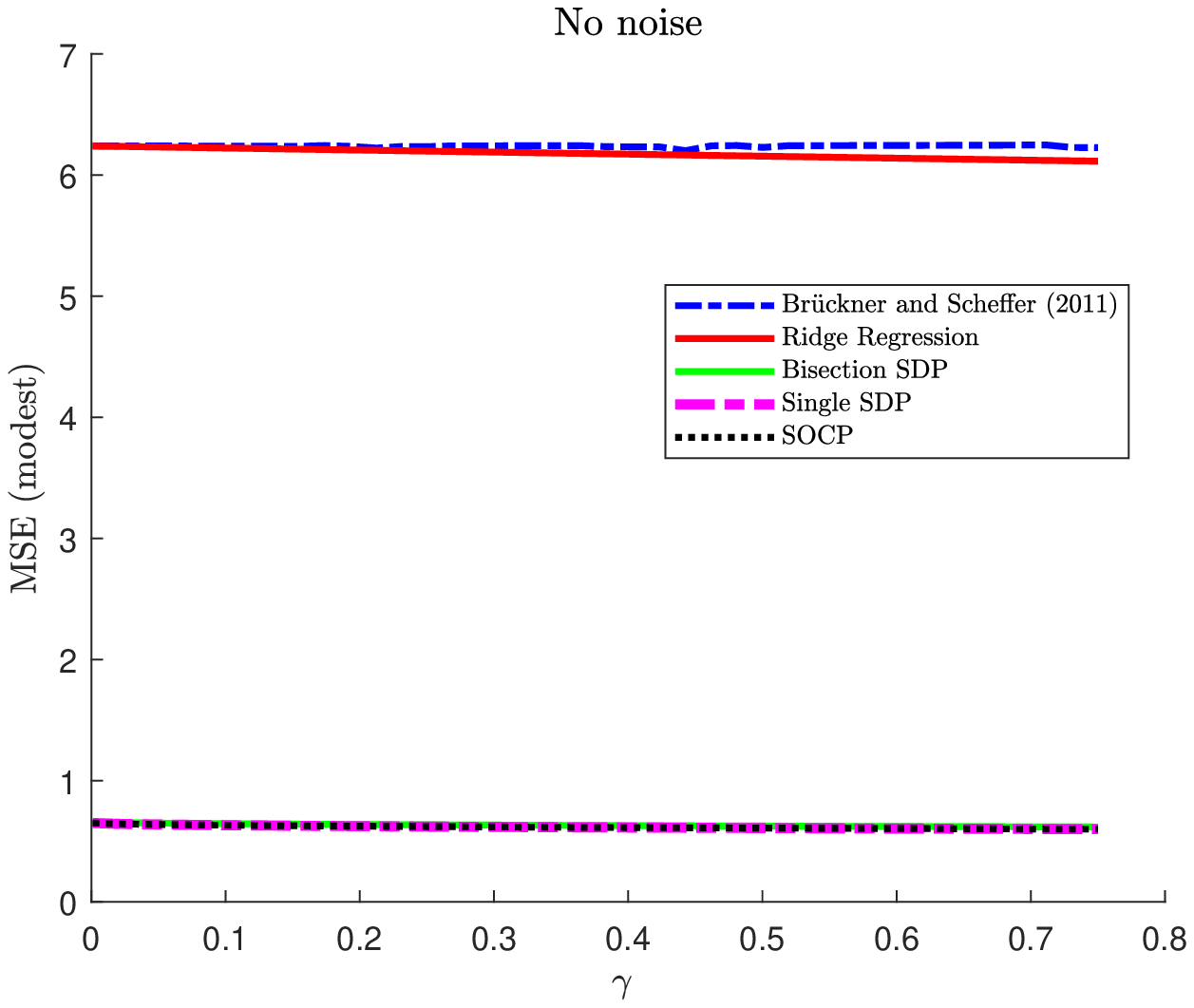}}
\subfigure{\includegraphics[width=0.32\textwidth]{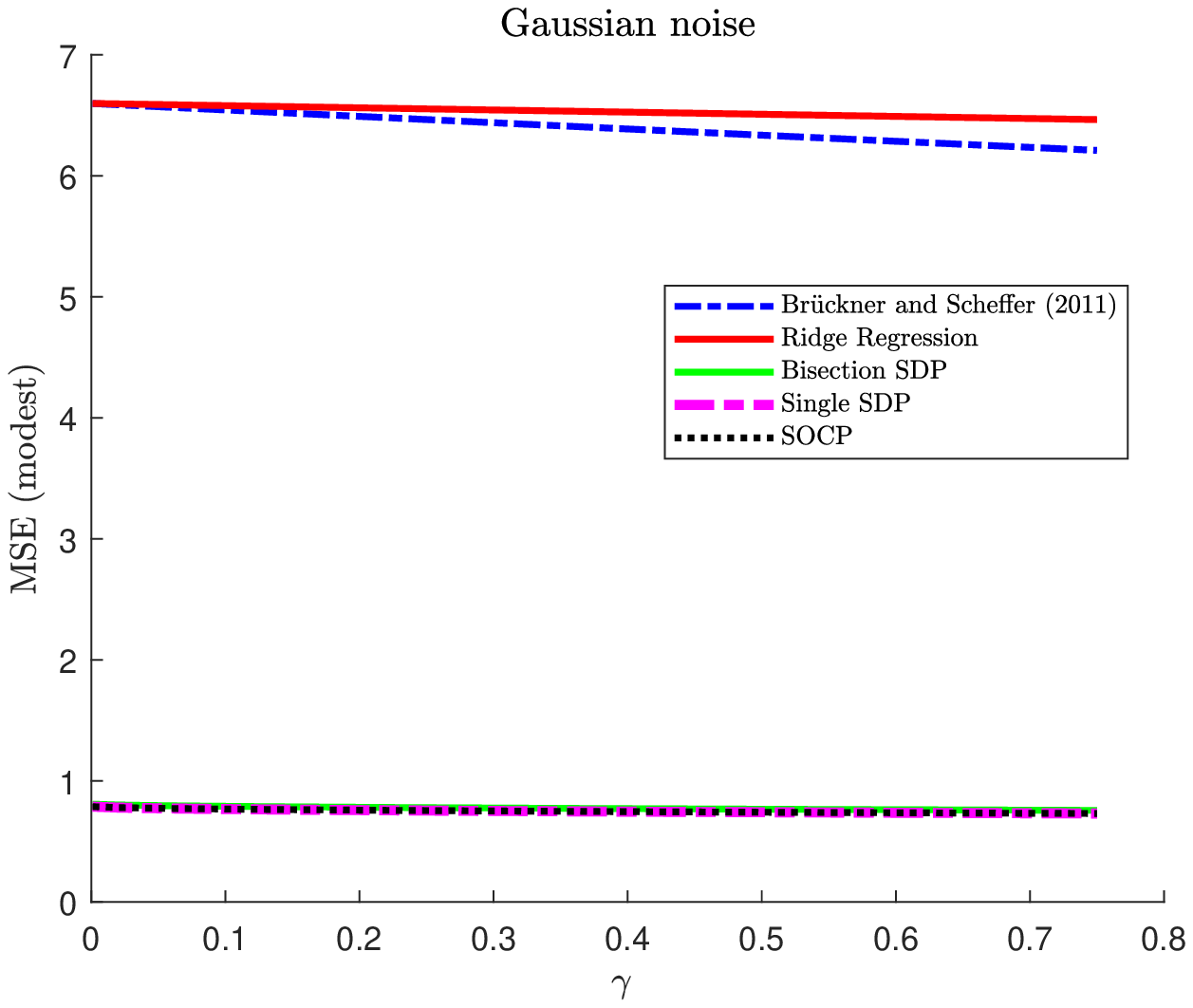}}
\subfigure{\includegraphics[width=0.32\textwidth]{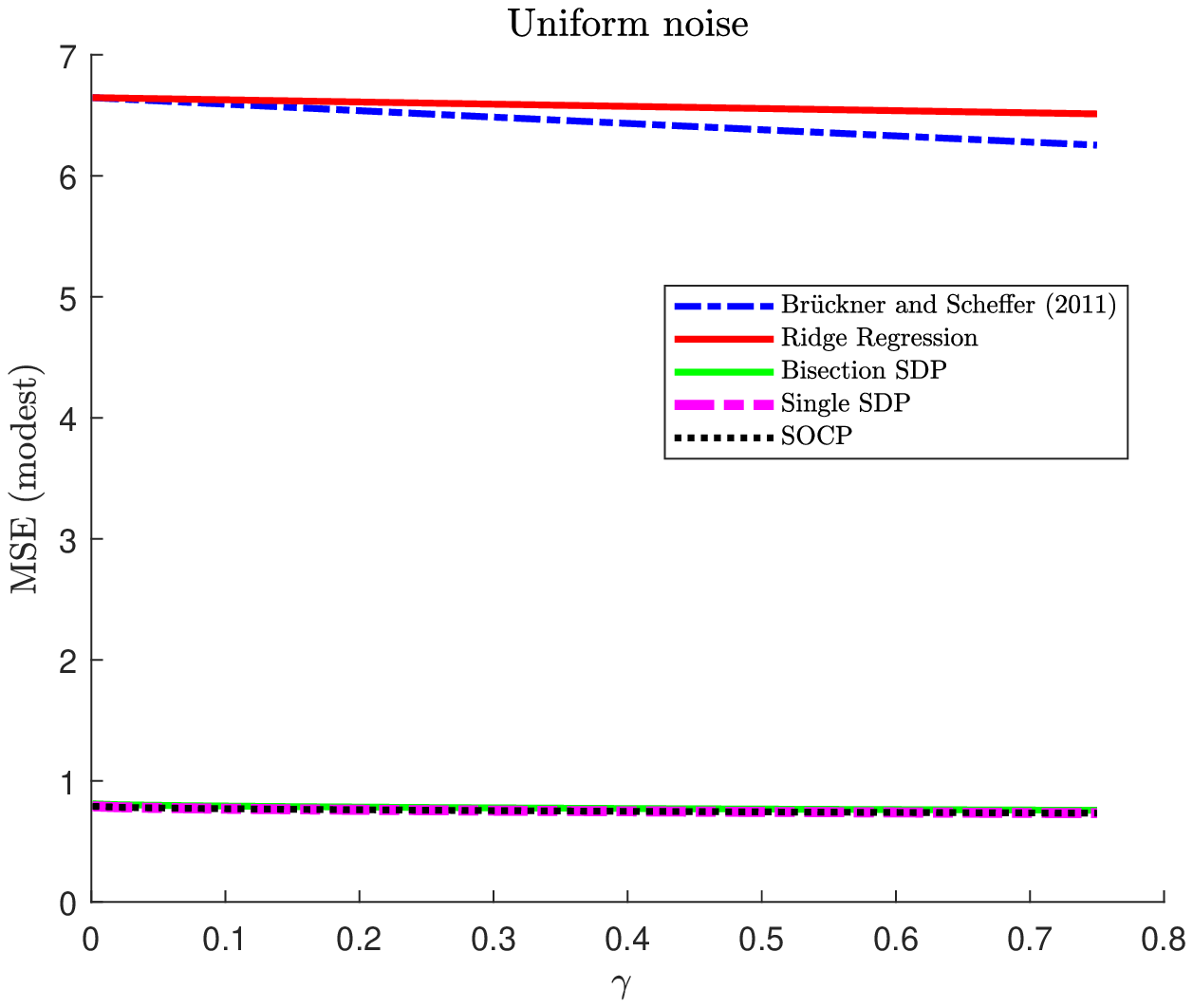}}
\vspace{-5mm}
\caption{MSE comparison for wine dataset generated by $\mathcal A_{\rm severe}$.}
\label{fig:noisewinesevere}
\vspace{-4mm}
\end{figure}

\end{document}